\documentclass[a4paper,11pt]{amsart}
\usepackage[left=25mm,top=30mm,bottom=25mm,right=25mm]{geometry}
\usepackage[utf8]{inputenc}
\usepackage{amssymb}
\usepackage{cmap}
\usepackage{hyperxmp}
\usepackage[pdfdisplaydoctitle = true,
   colorlinks = true,
   urlcolor = blue,
   citecolor = blue,
   linkcolor = blue,
   pdfstartview = FitH,
   pdfpagemode = UseNone,
   bookmarksnumbered = true]{hyperref}
\usepackage{color}
\usepackage[all]{xy}

\theoremstyle{plain}
\newtheorem{thm}{Theorem}[section]
\newtheorem{cor}[thm]{Corollary}
\newtheorem{lem}[thm]{Lemma}
\newtheorem{prop}[thm]{Proposition}
\newtheorem*{thm*}{Theorem}

\theoremstyle{definition}
\newtheorem{defn}[thm]{Definition}

\theoremstyle{remark}
\newtheorem{rem}[thm]{Remark}

\numberwithin{equation}{section}

\newcommand{\NN}{\mathbb{N}} 
\newcommand{\RR}{\mathbb{R}} 
\newcommand{\CC}{\mathbb{C}} 

\newcommand{\Rd}{\mathbb{R}^{d}} 

\newcommand{\id}{\mathrm{id}}

\newcommand{\Vect}{\mathrm{Vect}} 
\newcommand{\HMd}[1]{H^{#1}(M,\mathbb{R}^{d})} 
\newcommand{\HR}[1]{H^{#1}(\mathbb{R}^{d},\mathbb{R})} 
\newcommand{\HRd}[1]{H^{#1}(\mathbb{R}^{d},\mathbb{R}^{d})} 
\newcommand{\CScR}{C_{c}^{\infty}(\mathbb{R}^{d},\mathbb{R})} 
\newcommand{\ScR}{\mathcal{S}(\mathbb{R}^{d})} 

\newcommand{\OPS}[1]{\Psi^{#1}}

\newcommand{\Diff}{\mathrm{Diff}}
\newcommand{\DiffM}{\mathrm{Diff}^{\infty}(M)} 
\newcommand{\SDiffM}{\mathrm{SDiff}_{\mu}^{\infty}(M)} 
\newcommand{\DiffRd}{\mathrm{Diff}_{H^{\infty}}(\mathbb{R}^{d})} 
\newcommand{\DiffTd}{\mathrm{Diff}^{\infty}(\mathbb{T}^{d})} 
\newcommand{\DM}[1]{\mathcal{D}^{#1}(M)} 
\newcommand{\DRd}[1]{\mathcal{D}^{#1}(\mathbb{R}^{d})} 

\newcommand{\norm}[1]{\left\Vert#1\right\Vert}
\newcommand{\abs}[1]{\left\vert#1\right\vert}
\newcommand{\set}[1]{\left\{#1\right\}}

\DeclareMathOperator{\ad}{ad} %
\DeclareMathOperator{\Op}{Op} %
\DeclareMathOperator{\Rec}{Rec} %
\DeclareMathOperator{\dive}{div} %
\DeclareMathOperator{\pr}{pr} %

\hypersetup{
 pdftitle = {Well-posedness of the EPDiff equation with a pseudo-differential inertia operator}, 
 pdfauthor = {M. Bauer, M. Bruveris, E. Cismas, J. Escher and B. Kolev}, 
 pdfsubject = {MSC 2010: 58D05, 35Q35}, 
 pdfkeywords = {EPDiff equation; diffeomorphism groups; Sobolev metrics of fractional order}, 
 pdflang = EN 
 }

\begin{document}

\title[EPDiff equation with pseudo-differential inertia operator]{Well-posedness of the EPDiff equation with a pseudo-differential inertia operator}

\author{M. Bauer}
\thanks{M. Bauer was partially supported by NSF-grant 1912037 (collaborative research in connection with NSF-grant 1912030) and E. Cismas was partially supported by CNCS UEFISCDI, project number PN-III-P4-ID-PCE-2016-0778.}
\address{Department of Mathematics, Florida State University, 32301 Tallahassee, USA}
\email{bauer@math.fsu.edu}

\author{M. Bruveris}
\address{Onfido, 3~Finsbury Avenue, London EC2M 2PA, UK}
\email{martins.bruveris@gmail.com}

\author{E. Cismas}
\address{"Politehnica" University of Timisoara, Piata Victoriei Nr. 2, 300006 Timisoara}
\email{emanuel.cismas@upt.ro}

\author{J. Escher}
\address{Institute for Applied Mathematics, University of Hanover, D-30167 Hanover, Germany}
\email{escher@ifam.uni-hannover.de}

\author{B. Kolev}
\address{LMT (ENS Paris-Saclay, CNRS, Université Paris-Saclay), F-94235 Cachan Cedex, France}
\email{boris.kolev@math.cnrs.fr}


\date{\today}

\begin{abstract}
  In this article we study the class of right-invariant, fractional order Sobolev-type metrics on groups of diffeomorphisms of a compact manifold $M$. Our main result concerns well-posedness properties for the corresponding Euler-Arnold equations, also called the EPDiff equations, which are of importance in mathematical physics and in the field of shape analysis and template registration. Depending on the order of the metric, we will prove both local and global well-posedness results for these equations. As a result of our analysis we will also obtain new commutator estimates for elliptic pseudo-differential operators.
\end{abstract}

\maketitle

\section{Introduction}
\label{sec:introduction}

Our goal in this article is to study the well-posedness of the EPDiff equation on the group of diffeomorphisms of a compact manifold $M$ for the $H^{s}$-metric when $s$ is no longer an integer. Our main result is the following.

\begin{thm*}[Local and Global Well-posedness]
  Let $M$ be a closed manifold of dimension $d$. The EPDiff equation and the geodesic equation for the fractional, right invariant $H^{s}$-metric on the diffeomorphism group $\DiffM$ are locally well-posed, provided $s \ge 1/2$. For $s>\frac{d}2+1$ all solutions exist for all time $t$, i.e., the equations are globally well-posed.
\end{thm*}

Our result is proven under more general assumptions, namely for right invariant metrics that are defined using  abstract pseudo differential operators as inertia operator. This will allow us to apply the result to a wide class of situations, including in particular the fractional $H^s$-metric, but also other examples such as the so-called information metric as studied in~\cite{Mod2014}. Our result requires us to carefully investigate smoothness properties of conjugating pseudo-differential operators by diffeomorphisms. As a byproduct of our analysis, we obtain the following result which is of independent interest for the study of pseudo-differential operators and can be viewed as generalized, higher order Kato--Ponce type inequalities.

\begin{thm*}[Smooth Conjugation of Pseudo Differential Operators]
  Let $A$ be a pseudo-differential operator in $\OPS{r}(M)$ with $r \geq 1$. Then the map
  \begin{equation*}
    \DM{q} \to L(H^{q}(M), H^{q-r}(M))\,,\quad
    \varphi \mapsto A_{\varphi}
  \end{equation*}
  is smooth for $q \geq r$ and $q > d/2 + 1$.
\end{thm*}

The operator $A_{\varphi}$ is also called the {twisted map}, i.e., the inertia operator $A$ twisted by the right translation $R_{\varphi}$:
\begin{equation*}
  A_{\varphi} : = R_{\varphi}\circ A\circ R_{\varphi^{-1}},
\end{equation*}
with $R_{\varphi} v : = v \circ \varphi$ for $\varphi\in\DM{q}$ being an element in the group of diffeomorphisms of Sobolev order $q$ and $v\in \mathfrak X^q(M)$, see Section~\ref{sec:Sobdiff}. For the exact
definition of the operator class $\OPS{r}(M)$, we refer to Section~\ref{sec:pseudo-diffs} and to the survey of Agranovich in~\cite{ES1994}.

\subsection*{Context of the result}

In 1765, Euler published a seminal paper~\cite{Eul1765a} in which he recast the equations of motion of a free rigid body as the geodesic flow on the rotation group. For the bi-century of this achievement, Arnold has extended this geometric framework to hydrodynamics and recast the equations of motion of a perfect fluid (with fixed boundary) as the geodesic flow on the volume-preserving diffeomorphisms group of the domain. Since, then a similar geometric formulation has been found for several important PDEs in mathematical physics, including in particular the Camassa--Holm equation \cite{CH1993,Mis1998,Kou1999}, the modified Constantin--Lax--Majda equation~\cite{CLM1985,Wun2010,EKW2012,BKP2016} or the SQG-equation~\cite{CMT1994,Was2016,BHP2018}, see~\cite{TV2011,KW2009} for further examples and references.

From a geometrical view-point, this theory can be reduced to the study of right-invariant Riemannian metrics on the diffeomorphism group of a manifold $M$ (or one of its subgroup like $\SDiffM$, the group of diffeomorphism which preserve a volume form $\mu$). To define a right invariant Riemannian metric on the diffeomorphism group $\Diff(M)$ of a compact Riemannian manifold $M$, it suffices to prescribe an inner product on its Lie algebra $\Gamma(TM)$. We will moreover assume that this inner product can be written as
\begin{equation*}
  \langle u_{1},u_{2}\rangle : = \int_{M} \left( Au_{1}\cdot u_{2} \right) d\mu \,,
\end{equation*}
where $u_{1}, u_{2} \in \Gamma(TM)$, $\cdot$ means the Riemannian metric on $TM$, $d\mu$, the Riemannian density and the inertia operator
\begin{equation*}
  A : \Gamma(TM) \to \Gamma(TM)
\end{equation*}
is a $L^{2}$-symmetric, positive definite, continuous linear operator. By translating this inner product, we get an inner product on each tangent space $T_{\varphi}\Diff(M)$, which is given by
\begin{equation}\label{eq:definition-metric}
  G_{\varphi}(v_{1},v_{2}) = \int_{M} \left( A_{\varphi} v_{1} \cdot v_{2} \right) \varphi^{*} d\mu \,,
\end{equation}
where $v_{1},v_{2}\in T_{\varphi}\Diff(M)$.

A geodesic for the metric $G$ on $\DiffM$ is an extremal curve $\varphi(t)$ of the energy functional
\begin{equation*}
  E(\varphi) : = \frac{1}{2} \int_{0}^{1} G_{\varphi}(\varphi_{t},\varphi_{t}) \, dt,
\end{equation*}
where subscript $t$ in $\varphi_{t}$ means time derivative. Let $u(t) : = R_{\varphi^{-1}(t)}\varphi_{t}(t)$ be the Eulerian velocity of the geodesic curve $\varphi(t)$. Then $u(t)$ is a solution of the {Euler-Poincaré equation (EPDiff) \cite{HM2005} on $\Diff(M)$}:
\begin{equation}\label{eq:EPDiff}
  m_{t} + \nabla_{u}m + \left(\nabla u\right)^{t}m + (\dive u) m = 0, \quad m : = Au \,,
\end{equation}
where $\left(\nabla u\right)^{t}$ is the Riemannian adjoint (for the metric on $M$) of $\nabla u$.
When $A$ is invertible, the EPDiff equation~\eqref{eq:EPDiff} can be rewritten as
\begin{equation}\label{eq:Diff-Euler-Arnold}
  u_{t} = - A^{-1}\left\{ \nabla_{u}Au + \left(\nabla u\right)^{t}Au + (\dive u) Au \right\},
\end{equation}
which is the {Euler--Arnold equation} for $\Diff(M)$. Important examples for the inertia operator $A$ include (fractional) powers of the Laplacian, which give then rise to the afore mentioned PDEs as corresponding geodesic equations.

As acknowledged by Arnold himself, his seminal paper concentrated on the {geometrical ideas} and not on the analytical difficulties that are inherent when {infinite dimensional manifolds} are involved. In 1970, Ebin \& Marsden~\cite{EM1970} reconsidered this geometric approach from the analytical point of view, see also~\cite{EMF1972,Shn1985,Bre1989,Shn1994,Bre1999,Che2010}. They proposed to look at the {Fréchet Lie group} of smooth diffeomorphisms as an {inverse limit of Hilbert manifolds}, following some ideas of Omori~\cite{Omo1970,Omo1997}. The remarkable observation is that, in this framework, the Euler equation (a PDE) can be recast as an ODE (the geodesic equation) on these Hilbert manifolds. Furthermore, following their approach, if we can prove {local existence and uniqueness of the geodesics} (ODE), then the EPDiff equation~\eqref{eq:EPDiff} is {well-posed}.

The local well-posedness of the geodesic equation when the inertia operator $A$ is a {differential operator} has been implicitly solved in the seminal article of Ebin and Marsden~\cite{EM1970}, see also~\cite{Shk1998,Shk2000,CK2003,TY2005,Gay2009a,MP2010,MM2013,KM2003,KO1987}, and hence for $H^{k}$-metrics on diffeomorphism groups, where $k$ is an integer. This result has been extended to invariant metrics on several related spaces of mappings, such as spaces of immersions, Riemannian metrics and the Virasoro--Bott group, see~\cite{KO1987,BHM2011,BHM2013,BBHM2018,BM2018,BBK2018}. In a series of papers~\cite{EKW2012,EK2014,BEK2015,Kol2017}, the local and global well-posedness problem for the general EPDiff equation on $\DiffTd$ or $\DiffRd$ when the inertia operator is a {non-local Fourier multiplier} was solved. This applies, in particular, to every $H^{s}$-metrics on $\DiffTd$ or $\DiffRd$ when $s$ is no longer an integer. In this article we extend this analysis to the EPDiff equation on compact manifolds, which requires us to deal with inertia operators which are general Pseudo Differential operators. Simultaneously to this article, the first author and collaborators proved in \cite{BHM2019} local well-posedness of geodesic equations for fractional order metrics on the space of immersions of a manifold $M$ with values in another manifold $N$. The class of operators studied in \cite{BHM2019} is defined via holomorphic functional calculus of the Laplace operator. In the special case of $M$ being $N$ their results agree with the first part of the main theorem of the present article (the local well-posedness of the geodesic equations), albeit for a different class of inertia operators and using a different method of proof (our strategy is heavily based on the group structure of $\DM{q}$ and is valid for general abstract pseudo differential operators, see the comments below).

\subsection*{Strategy of the Proof}

Our main theorem will follow as a direct consequence from the more general results in Section~\ref{sec:globalwellposedness_smooth}, where the equivalent result is shown for metrics with inertia operator a general elliptic Pseudo-differential operator. Our strategy to obtain this result is, following the seminal approach of Ebin--Marsden, based on extending the metric and spray to a Sobolev completion of the group of smooth diffeomorphisms, which will allow us to view the geodesic equation as an ODE. The main obstacle to obtain this result is to prove the smoothness of the conjugation of elliptic Pseudo-differential operators with diffeomorphisms of Sobolev order. From this result our main theorem, local and global well-posedness of the EPDiff equation, follows essentially using the same techniques as for integer order metrics. We will first prove our results for Pseudo-Differential operators on $M = \mathbb R^{d}$, which will involve explicit estimates for the $n$-th derivative and will consist the main technical part of the article. We will then extend the result to Pseudo-differential operators acting on manifolds by carefully using Whitney's embedding theorem and thus reducing it to the Euclidean case.

\subsection*{Outline}

In Section~\ref{sec:notations}, we will introduce the basic notations and recall several standard results on multiplication and composition in Sobolev spaces. The exact class of Pseudo-differential operators, that we are studying in this article, is presented in Section~\ref{sec:pseudo-diffs}. In Section~\ref{sec:smoothness-conjugates-pseudo-diff}, we will study the smoothness of the conjugation of Pseudo-differential operators and in Section~\ref{sec:metric-and-spray}, we will use this result to show that both the metric and the geodesic spray extend smoothly to groups of Sobolev diffeomorphisms. Finally, in Section~\ref{sec:well-posedness}, the previously developed theory will allow us to obtain our main results on local and global well-posedness of the EPDiff equations. Appendix~\ref{sec:Sobolev-boundedness} contains technical estimates, that were necessary for the derivation of the results in Section~\ref{sec:smoothness-conjugates-pseudo-diff}.

\subsection*{Acknowledgements}

We would like to thank Philipp Harms, Peter W. Michor, Gerard Misiolek, Klas Modin and Stephen C. Preston for fruitful discussions during the preparation of this manuscript.

\section{Notations and background material}
\label{sec:notations}

In this paper, we consider the group $\DiffM$ of smooth diffeomorphisms of a closed manifold $M$ of dimension $d$ which are isotopic to the identity. We equip this manifold with a Riemannian metric $g$ and let us denote by $\exp$ the Riemannian exponential map on $M$. $\DiffM$ can be endowed with a \emph{Fréchet-Lie group} structure modeled on the Fréchet vector space $\Gamma(TM)$, the space of smooth vector fields on $M$. A parametrization in a neighborhood of the identity is given by the mapping:
\begin{equation}\label{eq:localchart}
  \zeta : U_{0} \subset \Gamma(TM) \to V_{\id} \subset \DiffM,
\end{equation}
defined as:
\begin{equation}
  X \in \Gamma(TM)\mapsto \zeta(X), \qquad \zeta(X)(p) : = \exp_p(X(p)).
\end{equation}
The tangent space $T_{\varphi}\DiffM$ can be identified with the space $\Gamma(\varphi^{*}TM)$ of smooth sections above $\varphi$:
\begin{equation*}
  T_{\varphi}\DiffM = \set{X_{\varphi}\in\mathrm{C}^{\infty}(M,TM);\; \pi \circ X_{\varphi}(p) = \varphi(p)},
\end{equation*}
where $\pi: TM\to M$ is the canonical projection.

The Fréchet-Lie group $\DiffM$ has the Lie algebra $\Gamma(TM)$, the space of smooth vector fields on $M$, with the Lie algebra bracket:
\begin{equation*}
  \ad_{u}v : = -[u,v], \qquad u,v\in\Gamma(TM),
\end{equation*}
the negative of the standard Lie bracket of vector fields. Since moreover $M$ is compact, $\DiffM$ is a \emph{regular Fréchet Lie group} in the sense of Milnor~\cite{Mil1984}. In particular, each element $u$ of the Lie algebra $\Gamma(TM)$, corresponds to a one-parameter subgroup of $\DiffM$.

The regular dual of $\Gamma(TM)$ is identified with $\Gamma(TM)$ via the pairing:
\begin{equation*}
  (m,u) = \int_{M} (m \cdot u)\, d\mu, \qquad m, u\in \Gamma(TM).
\end{equation*}

We will also be interested in the diffeomorphism group of $\Rd$. But, since difficulties arise due to the non-compactness of $\Rd$, we cannot use the full group of smooth diffeomorphisms but need to restrict our study to some subgroup with nice behaviour at infinity. We will set:
\begin{equation*}
  \DiffRd : = \set{\id + u;\; u \in \HRd{\infty} \; \text{and} \; \det(\id + du)> 0}\,,
\end{equation*}
where $ \HRd{\infty}$ denotes the space of $\Rd$-valued $H^{\infty}$-functions on $\Rd$, i.e.,
\begin{equation*}
  \HRd{\infty} : = \bigcap_{q \ge 0} \HRd{q}\,,
\end{equation*}
and where $\HRd{q}$ denotes the ($\Rd$-valued) Sobolev space on $\Rd$, defined below.

Let $\mathfrak{F}$ be the Fourier transform on $\RR^{d}$, defined with the following normalization
\begin{equation*}
  \hat{f}(\xi) = (\mathfrak{F} f)(\xi) = \int_{\Rd} e^{-2i\pi \langle x,\xi\rangle} f(x) \, dx
\end{equation*}
where $\xi$ is the independent variable in the frequency domain. With this convention, its inverse $\mathfrak{F}^{-1}$ is given by:
\begin{equation*}
  (\mathfrak{F}^{-1} g)(x) = \int_{\Rd} e^{2i\pi \langle x,\xi\rangle} g(\xi) \, d\xi \, .
\end{equation*}
For $q\in \RR^{+}$ the Sobolev $H^{q}$-norm of a function $f$ on $\Rd$ is
defined by
\begin{equation*}
  \norm{f}_{H^{q}}^{2} : = \norm{\langle \xi \rangle^{q} \hat{f}}_{L^{2}}^{2}\, ,
\end{equation*}
where
\begin{equation*}
  \abs{\xi} : = (\xi_{1}^{2} + \dotsb + \xi_{d}^{2})^{1/2},
\end{equation*}
and
\begin{equation*}
  \langle \xi \rangle : = (1 + \abs{\xi}^{2})^{1/2}.
\end{equation*}

The Sobolev spaces $\HR{q}$ is defined as the closure of the space of compactly supported functions, $C_{c}^{\infty}(\Rd,\RR)$, relatively to this norm and the space $\HRd{q}$ is the space of $\Rd$-valued functions of which each component belongs to $\HR{q}$.

Following \cite[Sect.~7.2.1]{Tri1992} we will now introduce the space $H^{q}(M,\RR)$, of functions of Sobolev class $H^{q}$ on a closed $d$-dimensional Riemannian manifold $(M,g)$. Denote by $B_\epsilon(x)$ the ball of radius $\epsilon$ with center $x$. We can choose a finite cover of $M$ by balls $B_\epsilon(x_\alpha)$ with $\epsilon$ sufficiently small, such that normal coordinates are defined in the ball $B_\epsilon(x)$, and a partition of unity $\varrho_\alpha$, subordinated to this cover. Using this cover, we define the $H^{q}$-norm of a function $f$ on $M$ via
\begin{align*}
  \norm{f}_{H^{q}(M,g)}^{2} & = \sum_{\alpha} \norm{(\varrho_\alpha f)\circ \exp_{x_\alpha}}^{2}_{H^{q}(\RR^{d})}                                         \\
                            & = \sum_{\alpha} \norm{\langle \xi \rangle^{q} \mathcal{F} ((\varrho_\alpha f)\circ \exp_{x_\alpha})}^{2}_{L^{2}(\RR^{d})} .
\end{align*}
Changing the cover or the partition of unity leads to equivalent norms,
see \cite[Theorem 7.2.3]{Tri1992}. When $q$ is an integer, we get norms which are equivalent to the Sobolev norms treated
in~\cite[Chapter 2]{Eic2007}. The norms depend on the choice of the Riemann metric $g$, but different choices of metrics lead to again to equivalent norms. The dependence on the metric is worked out in detail in \cite{Eic2007}. For functions with values in a vector bundle we use a (local) trivialization and define the norm in each coordinate as above. This leads (up to equivalence) to a well-defined $H^{q}$-norm for functions with values in a vector bundle.

\subsection{Groups of diffeomorphisms of finite regularity}
\label{sec:Sobdiff}

$\DiffRd$ has a stronger structure than just a \emph{Fréchet Lie group}, it is the inverse limit of Hilbert manifolds which are themselves topological groups (\emph{ILH-Lie groups} following Omori~\cite{Omo1997}). In other words,
\begin{equation*}
  \DiffRd = \bigcap_{q > 1 + d/2} \DRd{q},
\end{equation*}
where the set $\DRd{q}$ is defined, for $q>\frac{d}2+1$, as follows:
\begin{equation*}
  \DRd{q} : = \set{\id + u;\; u \in \HRd{q} \; \text{and} \; \det(\id + du)> 0}\,.
\end{equation*}
Note, that the condition $\det(\id + du)> 0$ is well-defined as for $q>\frac{d}{2}+1$ we have $\HRd{q}\subset C^1(\RR^{d},\RR^{d})$, c.f.
Lemma~\ref{lem:sobolev-embedding}. The set $\DRd{q}$ is a Hilbert manifold, modelled on $\HRd{q}$.

Similarly, we can introduce the Hilbert space $\Gamma^{q}(TM)$ of vector fields on $M$ of class $H^{q}$. For $q > 1 + d/2$, we define the set $\DM{q}$ of $C^{1}$ diffeomorphisms $\varphi$ of $M$, isotopic to the identity, and which are of class $H^{q}$. $\DM{q}$ is a smooth Hilbert manifold, modelled on $\Gamma^{q}(TM)$ and
\begin{equation*}
  \DiffM = \bigcap_{q > 1 + d/2} \DM{q}.
\end{equation*}

The Hilbert manifolds $\DRd{q}$ and $\DM{q}$ are topological groups (see~\cite{IKT2013}). They are however not Hilbert Lie groups, because composition and inversion are continuous but not smooth (see~\cite[Proposition 2.6]{IKT2013}). For a more detailed treatment of these manifolds, we refer to~\cite{IKT2013}.

\begin{rem}
  Note, that the tangent bundle $T\DRd{q}$ is a trivial bundle
  \begin{equation*}
    T\DRd{q}\cong\DRd{q}\times \HRd{q}\,,
  \end{equation*}
  because $\DRd{q}$ is an open subset of the Hilbert space $\HRd{q}$. Beware, however, that unless the manifold $M$ is parallelizable, the tangent bundle of the Hilbert manifold $\DM{q}$ is not trivial.
\end{rem}

\subsection{Sobolev embeddings, composition and multiplication theorems}

In this section, we will collect several results on composition and multiplication in Sobolev spaces that will be used throughout this paper. We start by recalling the following \emph{Sobolev embedding lemma} which proof can be found in~\cite[Proposition 2.2]{IKT2013} for $\RR^{d}$ and in~\cite{Aub1982} for a compact manifold.

\begin{lem}\label{lem:sobolev-embedding}
  Let $M$ be a closed manifold of dimension $d$, $q > d/2$ a real number and $k$ an integer. Then,
  \begin{enumerate}
    \item $\HMd{q+k}$ is continuously embedded into $C^{k}(M,\Rd)$;
    \item $\HRd{q+k}$ is continuously embedded into $C^{k}_{0}(\Rd,\Rd)$, the space of all $C^{k}$-functions vanishing at infinity.
  \end{enumerate}
\end{lem}

Next, we will recall the following result concerning the extension of pointwise multiplication to a bounded bilinear mapping between Sobolev spaces.

\begin{lem}\label{lem:pointwise-multiplication}
  Let $X$ be either $\RR^{d}$ or a closed manifold $M$ of dimension $d$. Let $q > d/2$ and $0 \le p \le q$ then pointwise multiplication extends to a bounded bilinear mapping
  \begin{equation*}
    H^{q}(X,\RR) \times H^{p}(X,\RR) \to H^{p}(X,\RR).
  \end{equation*}
  More precisely, there exists $C >0$ such that
  \begin{equation*}
    \norm{fg}_{H^{p}} \le C \norm{f}_{H^{q}} \norm{g}_{H^{p}},
  \end{equation*}
  for all $f \in H^{q}(X,\RR)$ and $g\in H^{p}(X,\RR)$. In particular $H^{q}(X,\RR)$ is a multiplicative algebra, if $q > d/2$.
\end{lem}

\begin{rem}
  For the proof of Lemma~\ref{lem:pointwise-multiplication} in the case of $\RR^{d}$, see~\cite[Lemma 2.3]{IKT2013}. In the case of a closed manifold $M$, it results from~\cite[Lemma 2.16]{IKT2013} using a partition of unity.
\end{rem}

Let $J_{\varphi}$ denotes the Jacobian determinant of a diffeomorphism $\varphi$ in $\mathcal{D}^{q}(X)$, where $X$ is either $\RR^{d}$ or a closed manifold $M$ of dimension $d$. From lemma~\ref{lem:pointwise-multiplication}, we deduce that the mapping
\begin{equation*}
  \varphi \mapsto J_{\varphi}, \qquad \mathcal{D}^{q}(X) \to H^{q-1}(X,\RR)
\end{equation*}
is smooth and we have moreover the following result, which is a reformulation of~\cite[Lemma 2.5]{IKT2013}.

\begin{lem}\label{lem:inverse-Jacobian-smoothness}
  Let $X$ be either $\RR^{d}$ or a closed manifold $M$ of dimension $d$. Let $q > 1 + d/2$ and $0 \le p \le q$. Given $\varphi \in \mathcal{D}^{q}(X)$ and $f \in H^{p}(X,\RR)$, the function $f/J_{\varphi}$ belongs to $H^{p}(X,\RR)$ and the mapping
  \begin{equation*}
    (\varphi,f) \mapsto \frac{f}{J_{\varphi}}, \qquad \mathcal{D}^{q}(X) \times H^{p}(X,\RR) \to H^{q}(X,\RR)
  \end{equation*}
  is smooth.
\end{lem}

Finally, we recall the following result concerning the right action of $\mathcal{D}^{q}(X)$ on $H^{p}(X,N)$, where $N$ is a manifold of dimension $d'$ (see~\cite[Lemma 2.7]{IKT2013} and~\cite[Proposition 3.10]{IKT2013}).

\begin{lem}\label{lem:composition-regularity}
  Let $X$ be either $\RR^{d}$ or a closed manifold $M$ of dimension $d$ and let $N$ be a manifold of dimension $d'$. Given any two real numbers $p,q$ with $q > 1 + d/2$ and $q \ge p\ge 0$, the mapping
  \begin{equation*}
    H^{p}(X,N) \times \mathcal{D}^{q}(X) \rightarrow H^{p}(X,N), \qquad (u,\varphi) \mapsto u \circ \varphi
  \end{equation*}
  is continuous. Moreover, the mapping
  \begin{equation*}
    R_{\varphi}: u \mapsto u \circ \varphi
  \end{equation*}
  is \emph{locally bounded}. More precisely, there exists a neighbourhood $U$ of $\id$ in $\mathcal{D}^{q}(X)$ and a constant $C >0$ such that
  \begin{equation*}
    \norm{R_{\varphi}}_{\mathcal{L}(H^{p},H^{p})} \leq C ,
  \end{equation*}
  for all $\varphi\in U$.
\end{lem}

\section{Pseudo-differential operators}
\label{sec:pseudo-diffs}

\subsection{Pseudo-differential operators on $\RR^{d}$}

Roughly speaking, a \emph{pseudo-differential operator} $A$, acting on scalar-valued functions on $\RR^{d}$ is a linear operator which can be written as
\begin{equation*}
  Au(x) = \int_{\RR^{d}} e^{2\pi i x \cdot \xi} a(x,\xi) \hat{u}(\xi) \, d\xi
\end{equation*}
where the function $a(x,\xi)\in \CC$ is called the \emph{symbol} of the pseudo-differential operator $A$ and such operators will be denoted by $\Op(a)$ or $a(x,D)$.

\begin{rem}
  Note that a pseudo-differential operator $A$ preserves real functions iff its symbol $a$ satisfies:
  \begin{equation*}
    a(x,-\xi) = \overline{a(x,\xi)}, \qquad \xi\in\RR^{d}.
  \end{equation*}
\end{rem}

Of course, some regularity conditions are required on the symbol $a$ to insure that the operator is well-defined on some kind of function space. In this paper, we will restrict to the following class of symbols.

\begin{defn}
  Given $r \in \RR$, we will say that $a \in \mathbf{S}^{r}(\RR^{d} \times \RR_{d})$ if $a = a(x,\xi)$ is smooth on $\RR^{d} \times \RR_{d}$, with values in $\mathrm{M}_{d}(\CC)$ and if
  \begin{equation*}
    \abs{\partial_{x}^{\beta}\partial_{\xi}^{\alpha}a(x,\xi)} \le C_{\alpha,\beta}\langle \xi \rangle^{r-\abs{\alpha}},
  \end{equation*}
  for each $\alpha, \beta \in \NN^{d}$, where $\abs{\alpha} : = \alpha_{1} + \dotsb + \alpha_{d}$ and the constants $C_{\alpha,\beta}$ do not depend on $(x,\xi)$. The class of pseudo-differential operators with symbol in $\mathbf{S}^{r}(\RR^{d} \times \RR_{d})$ will be denoted by $\OPS{r}(\RR^{d})$.
\end{defn}

Each operator $A$ from this class is well-defined on the Schwartz space $\ScR$, of rapidly decreasing functions and sends this space on itself. Moreover, by the $L^{2}$ boundedness theorem (see~\cite[Chapter 2]{RT2010} for instance), $A$ extends to a bounded operator from $\HR{q}$ to $\HR{q-r}$, for all $q \ge r$. In particular, such an operator defines a linear and continuous operator
\begin{equation*}
  A : \HR{\infty} \to \HR{\infty}.
\end{equation*}

The theory of pseudo-differential operators can be easily extended to $\RR^{k}$-valued functions. In that case, the symbol $a(x,\xi)$ is matrix-valued and belong to $M_{k}(\CC)$. In that case, we define similarly the class of symbols $\mathbf{S}^{r}(\RR^{d} \times \RR_{d},M_{k}(\CC))$ and the class of operators $\OPS{r}(\RR^{d},\RR^{k})$. We have, moreover, the following nice properties, concerning composition and commutators see~\cite[Theorem 1.2.4]{ES1994}.

\begin{lem}\label{lem:pseudodiff-commutators}
  Let $A\in \OPS{r_{1}}(\RR^{d},\RR^{k})$ and $B\in \OPS{r_{2}}(\RR^{d},\RR^{k})$, then:
  \begin{enumerate}
    \item $A\circ B\in \OPS{r_{1}+r_{2}}(\RR^{d},\RR^{k})$,
    \item $[A,B]\in \OPS{r_{1}+r_{2}-1}(\RR^{d},\RR^{k})$, if their principal symbols $a(x,\xi)$ and $b(x,\xi)$ commute.
  \end{enumerate}
\end{lem}

\begin{rem}
  In particular taking for $B$ the differential operator $D_{i} = \frac{\partial}{\partial x_{i}}$, we get that
  \begin{equation*}
    [A,D_{i}] \in \OPS{r_{1}}(\RR^{d},\RR^{k}),
  \end{equation*}
  and taking for $B$ the multiplication operator with some function $f\in \CScR$, we get that
  \begin{equation*}
    [A,f] \in \OPS{r_{1}-1}(\RR^{d},\RR^{k}).
  \end{equation*}
\end{rem}

In order to prove the existence and smoothness of the spray on the extended Hilbert manifolds $\DM{q}$, we will need an \emph{ellipticity condition} on the inertia operator $A$. For our purpose, we will adopt the following definition

\begin{defn}
  A pseudo-differential operator
  \begin{equation*}
    A = a(x,D) \in \OPS{r}(\RR^{d},\RR^{d})
  \end{equation*}
  is called \emph{elliptic} if its symbol $a(x,\xi)\in\mathrm{GL}(\CC^{d})$ and
  \begin{equation*}
    \norm{[a(x,\xi)]^{-1}} \lesssim \left( 1 + \abs{\xi}^{2}\right)^{-r/2},
  \end{equation*}
  for all $x, \xi \in \Rd$.
\end{defn}

\begin{rem}
  An elliptic pseudo-differential operator in $\OPS{r}(\RR^{d},\RR^{d})$ induces a \emph{bounded isomorphism} between $\HRd{q}$ and $\HRd{q-r}$ for all $q\in \RR$.
\end{rem}

We summarize our considerations by introducing the following class of inertia operators which will be denoted by $\mathcal{E}^{r}(\RR^{d})$.

\begin{defn}\label{def:class-Er}
  An operator $A\in \mathcal{L}(\HRd{\infty})$ is in the class $\mathcal{E}^{r}(\RR^{d})$ iff the following conditions are satisfied:
  \begin{enumerate}
    \item $A = a(x,D) \in \OPS{r}(\RR^{d},\RR^{d})$;
    \item $A = a(x,D)$ is elliptic;
    \item Its symbol, $a(x,\xi)$ is Hermitian and positive definite for all $\xi \in \Rd$.
  \end{enumerate}
\end{defn}

\subsection{Pseudo-differential operators on a vector bundle}

We shall first recall the definition of a pseudo-differential operator acting on functions defined on a closed orientable manifold $M$ of dimension $d$. We follow closely~\cite{Hoe2007a} (see also~\cite{ES1994}) and start with the way we pullback operators.

\begin{defn}
  Consider a chart $(\mathcal{U},\kappa)$ on $M$, with $\kappa: \mathcal{U} \to U$ (where $U$ is an open set in $\RR^{d}$), then the pullback $\kappa^{*} P:\mathrm{C}_c^{\infty}(\mathcal{U})\to \mathrm{C}^{\infty}(\mathcal{U}) $
  of a linear operator $P:\mathrm{C}_c^{\infty}(U)\to \mathrm{C}^{\infty}( U)$ is
  defined by:
  \begin{equation*}
    (\kappa^{*}P)f: = P(f\circ\kappa^{-1})\circ \kappa ,\hspace{0,5cm}f\in \mathrm{C}_c^{\infty}(\mathcal{U}).
  \end{equation*}
\end{defn}

\begin{defn}\label{def:Pscalar}
  A bounded linear operator $A: C^{\infty}(M) \to C^{\infty}(M)$ is a pseudo-differential operator in the class $\OPS{r}(M)$ if for every local chart $(\mathcal{U}, \kappa)$, with $\kappa:\mathcal{U}\to U\subset \RR^{d}$, there exists a pseudo-differential operator $A_{U}\in \OPS{r}(\RR^{d})$ such that if $\Phi, \Psi\in C_c^{\infty}(\mathcal{U})$ then:
  \begin{equation*}
    \Phi A \Psi f = \kappa^{*}(\varphi A_{U} \psi)f,
  \end{equation*}
  where $\varphi: = \Phi\circ \kappa^{-1}$, $\psi: = \Psi\circ \kappa^{-1}$, and $f\in C^{\infty}(M)$.
\end{defn}

\begin{rem}
  This means that the local representative of the function $\Phi A \Psi f$ is obtained applying the operator $\varphi A_{U}\psi$ to the local representative of $f$:
  \begin{equation*}
    (\Phi A \Psi f)\circ \kappa^{-1}(x) = \varphi A_{U}\psi (f\circ\kappa^{-1})(x), \ \ \ x\in U.
  \end{equation*}
\end{rem}

Consider now $E_{M},F_{M}$ two complex vector bundles on $M$, of ranks $d_{1}$ and $d_{2}$. We denote by:
\begin{equation*}
  \Phi_{\mathcal{U}}:E_{\mathcal{U}}\rightarrow \mathcal{U}\times \CC^{d_{1}},\quad \Psi_{\mathcal{U}}:F_{\mathcal{U}}\rightarrow \mathcal{U}\times \CC^{d_{2}},
\end{equation*}
the corresponding local trivialization over the same open set $\mathcal{U}\subset M$. We can pullback vector-valued functions, using $(\chi^{E})^{*}: C^{\infty}(U,\CC^{d_{1}}) \rightarrow C^{\infty}(\mathcal{U},E_{\mathcal{U}})$, defined as:
\begin{equation}\label{eq:pullback}
  (\chi^{E})^{*}u: = (\pr_{2}\circ \Phi_{\mathcal{U}})^{-1}( u\circ \kappa),\ \ \ \ u\in C^{\infty}(U,\CC^{d_{1}}),
\end{equation}
and push-forward local sections, by:
\begin{equation}\label{eq:pushforward}
  (\chi^{E})_*v: = \pr_{2}\circ \Phi_{\mathcal{U}}( v\circ \kappa^{-1}), \quad\quad v\in C^{\infty}(\mathcal{U},E_{\mathcal{U}}).
\end{equation}
In a similar manner, with the aforementioned scalar case, we will pullback linear operators $P: C_c^{\infty}(U,\CC^{d_{1}})\rightarrow C^{\infty}(U,\CC^{d_{2}})$ to obtain linear operators $\chi^{*}P:C_c^{\infty}(\mathcal{U},E_{\mathcal{U}})\rightarrow C^{\infty}(\mathcal{U},F_{\mathcal{U}})$, using:
\begin{equation*}
  ({\chi}^{*}P)v = (\chi^F)^{*}P((\chi^E)_{*}v),\qquad v\in C_c^{\infty}(\mathcal{U},E_{\mathcal{U}})\,.
\end{equation*}

\begin{defn}[Pseudo-differential operators on complex vector bundles]\label{def:P-complex}
  We say that a linear operator $A:\Gamma(E_{M})\rightarrow\Gamma(F_{M})$ is a pseudo-differential operator of class $\OPS{r}(M, E_{M}, F_{M})$ if it satisfies the conditions from the scalar case for $A_{U}$ with symbol in $\mathbf{S}^{r}(\RR^{d}\times \RR_{d},\mathrm{M}_{d_{2}\times d_{1}}(\mathbb C))$, and the pullback operator defined above.
\end{defn}

The tangent bundle $TM$ can be considered as a real subbundle of its complexification $TM^c: = TM\otimes \CC$. For the complex vector bundle $TM^c$ we pullback functions $u\in C^{\infty}(U,\CC^{d})$ using $(\chi^{TM^c})^{*}u: = (\pr_{2}\circ T\kappa^c)^{-1}(u\circ \kappa)$ and push-forward with $(\chi^{TM^c})_*v: = \pr_{2}\circ T\kappa^c(v\circ \kappa^{-1})$, where $T\kappa^c$ is the complexification of the tangent mapping $T\kappa$. By the way $TM$ embeds in $TM^c$ one has $\displaystyle T\kappa^c\big|_{TM} = T\kappa$.

Thus one can pullback linear operators to obtain operators on $C^{\infty}_c(\mathcal{U},T\mathcal{U})$. In order to define a pseudo-differential operator on a real vector bundle we have to assure that it sends real vector-valued functions to real vector-valued functions.

\begin{defn}[Elliptic Pseudo-differential operators on $TM$]
  We say that a linear operator $A:\Gamma(TM)\rightarrow\Gamma(TM)$ is a pseudo-differential operator of class $\OPS{r}(M,TM)$
  if for every local chart $(\mathcal{U}, \kappa)$, with $\kappa:\mathcal{U}\to U\subset \RR^{d}$, there exists a pseudo-differential operator $A_{U}$ with a Hermitian symbol in $\mathbf{S}^{r}(\RR^{d}\times \RR^{d},\mathrm{M}_{d}(\mathbb C))$ and such that for every pair of real functions $\Phi, \Psi\in C_c^{\infty}(\mathcal{U})$ one has:
  \begin{equation*}
    \Phi A \Psi v = \chi^{*}(\varphi A_{U} \psi)v,\qquad v\in \Gamma(M).
  \end{equation*}
  If all $A_{U}$ are elliptic, then we call the operator $A$ elliptic and we write $A\in \mathcal{E}^{r}(M,TM)$.
\end{defn}

\begin{rem}
  The properties of the local representatives $A_{U}$ having Hermitian symbols and being elliptic are preserved under a change of coordinates and thus this notion of elliptic pseudo-differential operators on $TM$ is well-defined. Furthermore, an elliptic pseudo-differential operator of class in $ \mathcal{E}^{r}(M,TM)$ induces a \emph{bounded isomorphism} between $H^{q}(M,TM)$ and $H^{q-r}(M,TM)$ for all $q\in \RR$.
\end{rem}

\section{Conjugates of pseudo-differential operators}
\label{sec:smoothness-conjugates-pseudo-diff}

In this part we will study the smoothness of conjugation for pseudo-differential operators. To prove the local well-posedness of the EPDiff equation we will need this result in the context of operators on vector bundles. We will however start by considering the problem in the more simpler situation of operators acting on functions on $\RR^{d}$.

\subsection{Conjugates of a pseudo-differential operator on $\RR^{d}$}
\label{subsec:smoothness-conjugates-Rd}

Let $A$ be a continuous linear operator from $\HRd{\infty}$ to itself and let
\begin{equation}\label{eq:twisted_map}
  A_{\varphi} : = R_{\varphi} \circ A \circ R_{\varphi^{-1}},
\end{equation}
where $R_{\varphi}v = v \circ \varphi$ and $\varphi \in \DiffRd$. Since $\DiffRd$ is a Fréchet Lie group with Lie algebra $\HRd{\infty}$, the mapping
\begin{equation*}
  (\varphi,v) \mapsto A_{\varphi}v, \qquad \DiffRd \times \HRd{\infty} \to \HRd{\infty}
\end{equation*}
is smooth. It could be interesting to note here that related observations have been made in~\cite[Proposition 1.3]{AS1971} and~\cite[Proposition 2.3]{MSS2013}. Nevertheless, these considerations are useless for our purpose, since we need a smoothness argument on Hilbert approximation manifolds. More precisely, the aim of this section is to prove the following theorem.

\begin{thm}\label{thm:smoothness-of-conjugates}
  Let $r \ge 1$ and $A = a(x,D) \in \OPS{r}(\RR^{d},\RR^{d})$ with a Hermitian symbol compactly supported in $x$. Then the mapping
  \begin{equation*}
    \varphi \mapsto A_{\varphi} : = R_{\varphi}A R_{\varphi^{-1}}, \quad \DRd{s} \to \mathcal{L}(\HRd{q},\HRd{q-r})
  \end{equation*}
  is smooth for $q > 1 + d/2$ and $q \ge r$.
\end{thm}

To solve this problem, it was observed in~\cite{EK2014} that the $n$-th partial (Gâteaux) derivative of~\eqref{eq:twisted_map}, in the smooth category, was given by:
\begin{equation*}
  \partial^{n}_{\varphi}A_{\varphi} (v,\delta\varphi_{1}, \dotsc ,\delta\varphi_{n}) = R_{\varphi}A_{n}R_{\varphi}^{-1}(v,\delta\varphi_{1}, \dotsc ,\delta\varphi_{n}),
\end{equation*}
where
\begin{equation*}
  A_{n}: = \partial^{n}_{\id}A_{\varphi} \in \mathcal{L}^{n+1}(\HRd{\infty},\HRd{\infty})
\end{equation*}
is the $(n+1)$-linear operator defined inductively by $A_{0} = A$ and
\begin{multline}\label{eq:recurrence-relation}
  A_{n+1}(u_{0},u_{1}, \dotsc , u_{n+1}) = \nabla_{u_{n+1}} \left(A_{n}(u_{0}, u_{1}, \dotsc , u_{n}) \right)
  \\
  - \sum_{k = 0}^{n} A_{n}(u_{0}, u_{1}, \dotsc ,\nabla_{u_{n+1}} u_{k}, \dotsc , u_{n}),
\end{multline}
where $\nabla$ is the canonical derivative on $\RR^{d}$. When $d = 1$, a nice formula for $A_{n}$ was obtained in~\cite{Cis2016a}. The strategy of the proof is then the same as the one explained in~\cite{EK2014,BEK2015,Kol2017}, which reduces the problem to show that each $A_{n}$ extends to a bounded $(n+1)$-linear operator from $\HRd{q}$ to $\HRd{q-r}$. More precisely, we have the following result, which will be stated without proof as it has already been proven in~\cite{EK2014,BEK2015,Kol2017}.

\begin{rem}\label{rem:fundamental-observation}
  In particular, for $n = 1$, we get
  \begin{equation*}
    \partial_{\varphi}A_{\varphi}(v,\delta\varphi) = R_{\varphi}A_{1}R_{\varphi}^{-1}(v,\delta\varphi), \quad \text{where} \quad A_{1}(u_{0},u_{1}) : = \left[ \nabla_{u_{1}},A \right]u_{0}.
  \end{equation*}
  Such a formula is still true if $(M,g)$ is a compact Riemannian manifold. In that case, the notation $\partial_{t}v$ should be replaced by $\overline{\nabla}_{\partial_{t}}v$, where the connection $\overline{\nabla}$ defined on $T\Diff(M)$ is induced by the connection $\nabla^{g}$ on $M$, see~\cite{FG1989,Bru2018,BHM2019}, and defined by
  \begin{equation*}
    \left(\overline{\nabla}_{\partial_{t}}v\right)(t,x) : = \nabla^{g}_{\varphi_{t}(t,x)}v(\cdot,x),
  \end{equation*}
  where $v(t)$ is a vector field on $T\Diff(M)$ defined along the curve $\varphi(t)$ on $\Diff(M)$. Furthermore, in that case, one can naturally construct a connection $\widetilde\nabla$ on the vector bundle $L(T\Diff(M),T\Diff(M))$ such that:
  \begin{equation*}
    \left(\widetilde{\nabla}_{\partial_{t}} A_{\varphi(t)}\right)v = \left[ \overline{\nabla}_{\partial_{t}},A_{\varphi(t)} \right]v.
  \end{equation*}
  Note, that in the notation of \cite{BHM2019} there is no distinction between the two covariant derivatives $\overline{\nabla}$ and $\widetilde{\nabla}$.
\end{rem}

\begin{lem}[Smoothness Lemma]\label{thm:smoothness_{l}emma}
  Let
  \begin{equation*}
    A : \HRd{\infty} \to \HRd{\infty}
  \end{equation*}
  be a continuous linear operator. Given $q > 1 + d/2$ with $q \ge r$, suppose that $A$ extends to a bounded linear operator from $\HRd{q}$ to $\HRd{q-r}$. Then
  \begin{equation*}
    \varphi \mapsto A_{\varphi}: = R_{\varphi} \circ A \circ R_{\varphi^{-1}} , \quad \DRd{s} \to \mathcal{L}(\HRd{q},\HRd{q-r})
  \end{equation*}
  is smooth, if and only if, each operator $A_{n}$ defined by \eqref{eq:recurrence-relation}, extends to a bounded $(n+1)$-linear operator in $\mathcal{L}^{n+1}(\HRd{s},\HRd{q-r})$.
\end{lem}
Therefore, the proof of Theorem~\ref{thm:smoothness-of-conjugates} reduces to prove the following result.

\begin{prop}\label{prop:An-boundedness}
  Let $r \ge 1$, $A = a(X,D) \in \OPS{r}(\RR^{d},\RR^{d})$ and $A_{n}$ be the $n$-linear operator defined inductively by~\eqref{eq:recurrence-relation}. Then, each $A_{n}$ extends to a bounded multi-linear operator
  \begin{equation*}
    A_{n} \in \mathcal{L}^{n+1}(\HRd{q},\HRd{q-r})
  \end{equation*}
  for $q > 1 + d/2$ and $q \ge r$.
\end{prop}

Before entering the details of the proof, it may be useful to rather think of $A_{n}$ as a $n$-linear mapping
\begin{equation*}
  \HRd{\infty} \times \dotsb \times \HRd{\infty} \to \mathcal{L}(\HRd{\infty},\HRd{\infty})
\end{equation*}
and write
\begin{equation*}
  A_{n}(u_{0},u_{1}, \dotsc , u_{n}) = A_{n}(u_{1}, \dotsc , u_{n})u_{0}.
\end{equation*}
The recurrence relation~\eqref{eq:recurrence-relation} rewrites then accordingly as:
\begin{multline}\label{eq:def-Rec}
  \Rec (A_{n})(u_{1}, \dotsc , u_{n+1}) : = [\nabla_{u_{n+1}}, A_{n}(u_{1}, \dotsc , u_{n})] \\
  - \sum_{k = 1}^{n} A_{n}(u_{1}, \dotsc ,\nabla_{u_{n+1}} u_{k}, \dotsc , u_{n}).
\end{multline}

\begin{rem}
  When $d = 1$ and $A$ commutes with $D : = d/dx$, the following nice formula for $A_{n}$ was obtained in~\cite{Cis2016a}:
  \begin{equation*}
    A_{n}(u_{1}, \dotsc , u_{n}) = [u_{1},[u_{2},[ \dotsb [ u_{n}, D^{n-1}A]\dotsb]]]D, \qquad n \ge 1 \,.
  \end{equation*}
\end{rem}

It may also be worth to recall the following general rules for commutators
\begin{equation}\label{eq:Leibniz}
  [AB,C] = A[B,C] + [A,C]B \qquad \text{(Leibniz identity)},
\end{equation}
and
\begin{equation}\label{eq:Jacobi}
  [A,[B,C]] + [B,[C,A]] + [C,[A,B]] = 0 \qquad \text{(Jacobi identity)}.
\end{equation}
Finally, we will introduce the following notations.

\begin{enumerate}
  \item Given $f_{1}, \dotsc , f_{n} \in \HR{\infty}$, the multiplication operator by $f_{1}\dotsb f_{n}$ will be denoted by
        \begin{equation*}
          M_{n}(f_{1}, \dotsc , f_{n}).
        \end{equation*}
  \item Given a linear operator $P$ on $\HRd{\infty}$ and a multi-index
        \begin{equation*}
          \alpha : = (\alpha_{1}, \dotsc , \alpha_{d}),
        \end{equation*}
        we define
        \begin{equation*}
          \ad_{D}^{\alpha} P : = \ad_{D_{1}}^{\alpha_{1}} \ad_{D_{2}}^{\alpha_{2}} \dotsb \ad_{D_d}^{\alpha_d}P
        \end{equation*}
        where
        \begin{equation*}
          \ad_{D_{j}}^{\alpha_{j}}P : = \underbrace{[D_{j}, [D_{j} \dotsb [D_{j}, P]\dotsb]}_\text{$\alpha_{j}$ times}, \quad j = \overline{1,d}.
        \end{equation*}
  \item Given a linear operator $P$ on $\HRd{\infty}$ and $f_{1}, \dotsc , f_{n} \in \HR{\infty}$, we define
        \begin{equation*}
          S_{n,P}(f_{1},f_{2}, \dotsc , f_{n}) : = [f_{1},[f_{2}, \dotsb [f_{n}, P]\dotsb]]
        \end{equation*}
        for $n \ge 1$ and $S_{0,P} : = P$.
\end{enumerate}

\begin{rem}
  It should be observed that $M_{n}$ and $S_{n,P}$ are $n$-linear and totally symmetric in $(f_{1}, \dotsc , f_{n})$. For $S_{n,P}$, this is due to the Jacobi identity~\eqref{eq:Jacobi}. Besides, any expression like $\ad_{D_{i_{1}}} \ad_{D_{i_{2}}} \dotsb \ad_{D_{i_{n}}}P$ can be rewritten as $\ad_{D}^{\alpha}P$ for some multi-index $\alpha$, by virtue of the Jacobi identity.
\end{rem}

Since the canonical connection on $\RR^{d}$ writes as
\begin{equation*}
  \nabla_{u} = \sum_{j = 1}^{d} u^{j}D_{j},
\end{equation*}
where $D_{j} : = \partial/\partial x^{j}$, we shall also introduce the following recurrence relation for linear operators $P_{n}(f_{1},f_{2}, \dotsc , f_{n})$ on $\HRd{\infty}$, depending linearly on $f_{1}, \dotsc , f_{n}$ in $\HR{\infty}$:
\begin{multline}\label{eq:def-Rec-k}
  \Rec^{j} (P_{n})(f_{1}, \dotsc , f_{n+1}) : = [f_{n+1}D_{j}, P_{n}(f_{1}, \dotsc , f_{n})] \\
  - \sum_{k = 1}^{n} P_{n}(f_{1}, \dotsc ,f_{n+1}\partial_{j} f_{k}, \dotsc , f_{n}).
\end{multline}

\begin{lem}\label{lem:Rec-Sn}
  For any $n\geq 0$ and $P \in \mathcal{L}(\HRd{\infty},\HRd{\infty})$, we have
  \begin{multline*}
    \Rec^{j}(S_{n,P})(f_{1},\dotsc , f_{n+1}) = M_{1}(f_{n+1}) \circ S_{n,[D_{j},P]}(f_{1},\dotsc , f_{n})
    \\
    + S_{n+1,PD_{j}}(f_{1},\dotsc , f_{n+1}) + S_{n,P}(f_{1},\dotsc f_{n}) \circ M_{1}(\partial_{j}f_{n+1}).
  \end{multline*}
\end{lem}

\begin{proof}
  The proof is based on the following two relations
  \begin{multline}\label{eq:generalized-Leibniz}
    S_{n+1,PD_{j}}(f_{1},\dotsc , f_{n+1}) = S_{n+1,P}(f_{1},\dotsc , f_{n+1})\circ D_{j}
    \\
    - \sum_{k = 1}^{n+1} S_{n,P}(f_{1},\dotsc , \hat{f}_{k}, \dotsc , f_{n+1}) \circ M_{1}(\partial_{j}f_{k})
  \end{multline}
  and
  \begin{multline}\label{eq:generalized-Jacobi}
    [D_{j},S_{n,P}(f_{1},\dotsc , f_{n})] = S_{n,[D_{j},P]}(f_{1},\dotsc , f_{n})
    \\
    + \sum_{k = 1}^{n} S_{n,P}(f_{1},\dotsc , \partial_{j}f_{k}, \dotsc , f_{n}),
  \end{multline}
  which proofs are direct by induction, using the Leibniz identity for the first one and the Jacobi identity for the second one. We have therefore
  \begin{multline*}
    \Rec^{j}(S_{n,P})(f_{1},\dotsc , f_{n+1}) = [f_{n+1}D_{j},S_{n,P}(f_{1},\dotsc , f_{n})]
    \\
    - \sum_{k = 1}^{n} S_{n,P}(f_{1},\dotsc , f_{n+1}\partial_{j}f_{k}, \dotsc , f_{n}),
  \end{multline*}
  which can be rewritten, using the Leibniz identity, as
  \begin{multline*}
    M_{1}(f_{n+1}) \circ [ D_{j} , S_{n,P}(f_{1},\dotsc , f_{n})] + [f_{n+1} , S_{n,P}(f_{1},\dotsc , f_{n})] \circ D_{j}
    \\
    - \sum_{k = 1}^{n} M_{1}(f_{n+1}) \circ S_{n,P}(f_{1},\dotsc , \partial_{j}f_{k}, \dotsc , f_{n})
    \\
    - \sum_{k = 1}^{n} S_{n,P}(f_{1},\dotsc , \hat{f}_{k}, \dotsc , f_{n+1}) \circ M_{1}(\partial_{j}f_{k}) .
  \end{multline*}
  Now, using~\eqref{eq:generalized-Leibniz} and \eqref{eq:generalized-Jacobi}, we obtain
  \begin{multline*}
    M_{1}(f_{n+1}) \circ S_{n,[D_{j},P]}(f_{1},\dotsc , f_{n}) + S_{n+1,PD_{j}}(f_{1},\dotsc , f_{n+1})
    \\
    + S_{n,P}(f_{1},\dotsc , f_{n}) \circ M_{1}(\partial_{j}f_{n+1}),
  \end{multline*}
  which achieves the proof.
\end{proof}

\begin{cor}\label{cor:An-splitting}
  Let $n \ge 1$. Then, $A_{n}(u_{1}, \dotsc , u_{n})$ is a sum of terms of the following two types, where $f_{i}$ stands for some component $u_{\sigma(i)}^{k_{i}}$ of $u_{\sigma(i)}$ and $\sigma$ is a permutation of $\set{1, \dotsc , n}$. The first type (Type I) writes
  \begin{equation}\label{eq:type1-terms}
    P_{n}^{1}(f_{1}, \dotsc , f_{n}) : = M_{n}(f_{1}, \dotsc , f_{n}) \circ \ad_D^{\alpha}P,
  \end{equation}
  where $P \in \OPS{r}(\RR^{d},\RR^{d})$. The second type (Type II) writes
  \begin{multline}\label{eq:type2-terms}
    P_{n}^{2}(f_{1}, \dotsc , f_{n}) : = M_{m_{1}}(f_{1}, \dotsc , f_{m_{1}}) \circ S_{m_{2},P}(f_{m_{1}+1}, \dotsc , f_{m_{1}+m_{2}}) \\
    \circ M_{m_{3}}(\partial_{p_{1}} f_{m_{1}+m_{2}+1}, \dotsc , \partial_{p_{m_{3}}}f_{m_{1}+m_{2}+m_{3}}) \circ D_{i},
  \end{multline}
  where $P \in \OPS{r+m_{2}-1}(\RR^{d},\RR^{d})$, and $m_{1}+m_{2}+m_{3} = n$.
\end{cor}

\begin{proof}
  We will prove the Lemma by induction on $n \ge 1$. For $n = 1$, we get
  \begin{equation*}
    A_{1}(u_{1}) = \sum_{j = 1}^{d} [u_{1}^{j}D_{j},A] = \sum_{j = 1}^{d} \left( u_{1}^{j}[D_{j},A] + [u_{1}^{j},A]D_{j}\right)
  \end{equation*}
  so we are done. Suppose now that the result holds for some $n \ge 1$, so that $A_{n}$ is a sum of terms of type I and II. Then $A_{n+1} = \Rec(A_{n})$ is a sum of terms $\Rec^{j}(P_{n}^{1})$ and $\Rec^{j}(P_{n}^{2})$ for $j = 1, \dotsc, d$. Observe, moreover, that if $n = p+q$ and
  \begin{equation*}
    P_{n}(f_{1}, \dotsc , f_{n}) = Q_{p}(f_{1}, \dotsc , f_{p}) \circ R_{q}(f_{p+1}, \dotsc , f_{n}),
  \end{equation*}
  then, due to the Leibniz rule~\eqref{eq:Leibniz}, we have
  \begin{equation}\label{eq:Rec-QR}
    \Rec^{j}(Q_{p}\circ R_{q}) = \Rec^{j}(Q_{p})\circ R_{q} + Q_{p}\circ \Rec^{j}(R_{q}).
  \end{equation}
  Now, a direct computation shows that $\Rec^{j}(M_{n}) = 0$ for all $n \ge 1$. We get thus
  \begin{equation*}
    \Rec^{j}(P_{n}^{1})(f_{1}, \dotsc , f_{n+1}) = M_{n}(f_{1}, \dotsc , f_{n}) \circ \Rec^{j}(\ad_D^{\alpha}P)(f_{n+1}).
  \end{equation*}
  But
  \begin{equation*}
    \Rec^{j}(\ad_D^{\alpha}P)(f_{n+1}) = f_{n+1}[D_{j},\ad_D^{\alpha}P] + [f_{n+1},\ad_D^{\alpha}P]D_{j},
  \end{equation*}
  and hence
  \begin{multline*}
    \Rec^{j}(P_{n}^{1})(f_{1}, \dotsc , f_{n+1}) = M_{n+1}(f_{1}, \dotsc , f_{n+1}) \circ \ad_D^{(\alpha_{1}, \dotsc , \alpha_{j} +1, \dotsc , \alpha_{d})}P
    \\
    + M_{n}(f_{1}, \dotsc , f_{n}) \circ S_{1,\ad_D^{\alpha}P}(f_{n+1}) \circ D_{j},
  \end{multline*}
  is a sum of operators of type I~\eqref{eq:type1-terms} and II~\eqref{eq:type2-terms} but of order $n+1$, because
  \begin{equation*}
    P \in \OPS{r}(\RR^{d},\RR^{d}) \implies \ad_D^{\alpha}P \in \OPS{r}(\RR^{d},\RR^{d}),
  \end{equation*}
  by Lemma~\ref{lem:pseudodiff-commutators}. Next, $\Rec^{j}(P_{n}^{2})(f_{1}, \dotsc , f_{n+1})$ is the sum of two terms. The first one
  \begin{multline*}
    M_{m_{1}}(f_{1}, \dotsc , f_{m_{1}}) \circ \Rec^{j}(S_{m_{2},P})(f_{m_{1}+1}, \dotsc , f_{m_{1}+m_{2}}, f_{n+1}) \\
    \circ M_{m_{3}}(\partial_{p_{1}} f_{m_{1}+m_{2}+1}, \dotsc , \partial_{p_{m_{3}}}f_{m_{1}+m_{2}+m_{3}}) \circ D_{i}
  \end{multline*}
  can be rewritten, due to Lemma~\ref{lem:Rec-Sn}, as
  \begin{align*}
     & M_{m_{1}+1}(f_{1}, \dotsc , f_{m_{1}}, f_{n+1}) \circ S_{m_{2},[D_{j},P]}(f_{m_{1}+1}, \dotsc , f_{m_{1}+m_{2}})
    \\
     & \quad \circ M_{m_{3}}(\partial_{p_{1}} f_{m_{1}+m_{2}+1}, \dotsc , \partial_{p_{m_{3}}}f_{m_{1}+m_{2}+m_{3}}) \circ D_{i}
    \\
     & + M_{m_{1}}(f_{1}, \dotsc , f_{m_{1}}) \circ S_{m_{2}+1,PD_{j}}(f_{m_{1}+1}, \dotsc , f_{m_{1}+m_{2}}, f_{n+1})
    \\
     & \quad \circ M_{m_{3}}(\partial_{p_{1}} f_{m_{1}+m_{2}+1}, \dotsc , \partial_{p_{m_{3}}}f_{m_{1}+m_{2}+m_{3}}) \circ D_{i}
    \\
     & + M_{m_{1}}(f_{1}, \dotsc , f_{m_{1}}) \circ S_{m_{2},P}(f_{m_{1}+1}, \dotsc , f_{m_{1}+m_{2}})
    \\
     & \quad \circ M_{m_{3}+1}(\partial_{p_{1}} f_{m_{1}+m_{2}+1}, \dotsc , \partial_{p_{m_{3}}}f_{m_{1}+m_{2}+m_{3}},\partial_{j}f_{n+1}) \circ D_{i}.
  \end{align*}
  The second one
  \begin{multline*}
    M_{m_{1}}(f_{1}, \dotsc , f_{m_{1}}) \circ S_{m_{2},P}(f_{m_{1}+1}, \dotsc , f_{m_{1}+m_{2}}) \\
    \circ M_{m_{3}}(\partial_{p_{1}} f_{m_{1}+m_{2}+1}, \dotsc , \partial_{p_{m_{3}}}f_{m_{1}+m_{2}+m_{3}}) \circ \Rec^{j}(D_{i})(f_{n+1}),
  \end{multline*}
  recasts as
  \begin{multline*}
    - M_{m_{1}}(f_{1}, \dotsc , f_{m_{1}}) \circ S_{m_{2},P}(f_{m_{1}+1}, \dotsc , f_{m_{1}+m_{2}}) \\
    \circ M_{m_{3}+1}(\partial_{p_{1}} f_{m_{1}+m_{2}+1}, \dotsc , \partial_{p_{m_{3}}}f_{m_{1}+m_{2}+m_{3}},\partial_{i}f_{n+1}) \circ D_{j},
  \end{multline*}
  since
  \begin{equation*}
    \Rec^{j}(D_{i})(f_{n+1}) = - M_{1}(\partial_{i}f_{n+1}) \circ D_{j}.
  \end{equation*}
  In both cases, these expressions are sums of operators of type I and II but of order $n+1$, because
  \begin{equation*}
    P \in \OPS{r+m_{2}-1}(\RR^{d},\RR^{d}) \implies [D_{j},P] \in \OPS{r+m_{2}-1}(\RR^{d},\RR^{d}),
  \end{equation*}
  and
  \begin{equation*}
    P \in \OPS{r+m_{2}-1}(\RR^{d},\RR^{d}) \implies PD_{j} \in \OPS{r+m_{2}}(\RR^{d},\RR^{d}),
  \end{equation*}
  by Lemma~\ref{lem:pseudodiff-commutators}. This achieves the proof.
\end{proof}

\begin{proof}[Proof of Proposition~\ref{prop:An-boundedness}]
  We have to show that each operator $A_{n}$ extends to a bounded operator in
  \begin{equation*}
    \mathcal{L}^{n}(H^{q}(\RR^{d},\RR^{d}),\mathcal{L}(\HRd{q},\HRd{q-r})).
  \end{equation*}
  By corollary~\ref{cor:An-splitting}, this reduces to show that each operator of type I~\eqref{eq:type1-terms} or II~\eqref{eq:type2-terms} extends to a bounded operator in
  \begin{equation*}
    \mathcal{L}^{n}(H^{q}(\RR^{d},\RR),\mathcal{L}(\HRd{q},\HRd{q-r})).
  \end{equation*}
  Let $f_{1}, \dotsc , f_{n} \in \CScR$ and $w \in \HRd{\infty}$. For an operator of type I, we get by Lemma~\ref{lem:pointwise-multiplication}
  \begin{equation*}
    \begin{split}
      \norm{P_{n}^{1}(f_{1}, \dotsc , f_{n})w}_{H^{q-r}} & \lesssim \norm{\prod\limits_{k = 1}^n f_{k}}_{H^{q}} \cdot \norm{\ad_D^{\alpha}Pw}_{H^{q-r}}
      \\
      & \lesssim \norm{f_{1}}_{H^{q}} \dotsb \norm{f_{n}}_{H^{q}}\cdot \norm{w}_{H^{q}},
    \end{split}
  \end{equation*}
  since $\ad_D^{\alpha}P \in \OPS{r}(\RR^{d},\RR^{d})$ and $\HR{q}$ is a multiplicative algebra ($q > 1 + d/2$). Consider now an operator of type II~\eqref{eq:type2-terms} and set
  \begin{equation*}
    W : = M_{m_{3}}(\partial_{p_{1}} f_{m_{1}+m_{2}+1}, \dotsc , \partial_{p_{m_{3}}}f_{m_{1}+m_{2}+m_{3}}) D_{i}w,
  \end{equation*}
  so that
  \begin{equation*}
    \norm{W}_{H^{q-1}} \lesssim \left(\prod_{k = 1}^{m_{3}} \norm{f_{{m_{1}+m_{2}+k}}}_{H^{q}}\right) \cdot \norm{w}_{H^{q}},
  \end{equation*}
  because $q > 1 + d/2$ and $\HR{q-1}$ is a multiplicative algebra. We get thus, by Lemma~\ref{lem:pointwise-multiplication} and Theorem~\ref{thm:boundedness-lemma}
  \begin{equation*}
    \begin{split}
      \norm{P_{n}^{2}(f_{1}, \dotsc , f_{n})w}_{H^{q-r}} & \lesssim \left(\prod_{k = 1}^{m_{1}} \norm{f_{k}}_{H^{q}}\right) \cdot \norm{S_{m_{2},P}(f_{m_{1}+1}, \dotsc , f_{m_{1}+m_{2}}) W }_{H^{q-r}}
      \\
      & \lesssim \left(\prod_{k = 1}^{m_{1}+m_{2}} \norm{ f_{k}}_{H^{q}} \right) \cdot \norm{W}_{H^{q-1}}
      \\
      & \lesssim \left(\prod_{k = 1}^{n} \norm{ f_{k}}_{H^{q}}\right) \cdot \norm{w}_{H^{q}},
    \end{split}
  \end{equation*}
  which achieves the proof, since $\CScR$ is dense in $\HR{q}$.
\end{proof}

\subsection{Conjugates of a pseudo-differential operator acting on functions}
\label{subsec:smoothness-conjugates-functions}

In this part, we aim to use the previously developed theory to prove the following theorem concerning the conjugation of a pseudo-differential operator $A$ acting on functions defined on a compact manifold $M$, by a diffeomorphism of $M$.

\begin{thm}\label{thm:conj-functions}
  Let $A$ be a pseudo-differential operator of class $\OPS{r}(M)$ with $r \geq 1$. Then the map
  \begin{equation*}
    \mathcal{D}^{q}(M) \to L(H^{q}(M), H^{q-r}(M))\,,\quad
    \varphi \mapsto A_{\varphi}
  \end{equation*}
  is smooth for $q \geq r$ and $q > d/2 + 1$.
\end{thm}

To prove this result we pursue the following strategy. We will use Whitney's embedding theorem to construct an embedding $\iota_{M} : M \to \RR^{d_{0}}$ for $d_{0}$ large enough. We will use this embedding to extend the pseudo-differential operator and the involved functions to the embedding space $\RR^{d_{0}}$ and thus reduce the result to the situation of the previous section.
Our construction will be based on extending all local representatives to pseudo-differential operators on $\RR^{d_{0}}$. In the second step, we will then glue these operators together to obtain a pseudo-differential operator on $\RR^{d_0}$ which is an extension of our operator on $M$. Here, by extension, we mean an operator that satisfies an equation similar to~\eqref{eq:extension} below. Therefore, we will consider first the case where $\RR^{d} \subset \RR^{d_{0}}$, with $d < d_{0}$ and $\RR^{d}$ is embedded in $\RR^{d_{0}}$ as the subspace $\RR^{d} \times \set{0}$.

\begin{lem}\label{lem:trRd}
  Let $A$ be a pseudo-differential operator of class $\OPS{r}(\RR^{d})$ with symbol $a$. Then, there exists a pseudo-differential operator $B$ of class $\OPS{r}(\RR^{d_{0}})$ such that:
  \begin{equation}\label{eq:extension}
    \iota^{\ast}_{\RR^{d}} B = A \iota^{\ast}_{\RR^{d}}\,,
  \end{equation}
  where $\iota_{\RR^{d}} : \RR^{d} \to \RR^{d_{0}}$ is the canonical embedding.
\end{lem}

\begin{proof}
  Denote the coordinates in $\RR^{d_{0}}$ by $x = (x', x'') \in \RR^{d} \times \RR^{d_{0}-d}$. Let $a(x',\xi')$ be the symbol of $A$ and define the symbol
  \begin{equation*}
    b(x,\xi) = b(x',x'',\xi',\xi'') : = a(x',\xi')\,.
  \end{equation*}
  It is clear that $b \in \mathbf{S}^{r}(\RR^{d_{0}} \times \RR_{d_{0}})$. Let $B = \Op(b)$.

  To verify that $\iota^{\ast}_{\RR^{d}} B = A \iota^{\ast}_{\RR^{d}}$ we start with the identity:
  \begin{equation*}
    \widehat{f(\cdot, x'')}(\xi') = \int_{\RR^{d_{0}-d}} e^{2\pi i x'' \cdot \xi''} \widehat f(\xi', \xi'') \,d \xi''\,,
  \end{equation*}
  and calculate:
  \begin{align*}
    B f(x', x'') & = \int_{\RR^{d_{0}}} e^{2\pi i x \cdot \xi} a(x', \xi') \widehat f(\xi) \,d\xi              \\
                 & = \int_{\RR^{d}} e^{2\pi i x' \cdot \xi'} a(x', \xi')
    \int_{\RR^{d_{0}-d}} e^{2\pi i x'' \cdot \xi''} \widehat f(\xi', \xi'') \,d\xi'' \, d\xi'                  \\
                 & = \int_{\RR^{d}} e^{2\pi i x' \cdot \xi'} a(x', \xi') \widehat{f(\cdot, x'')}(\xi') \,d\xi' \\
                 & = A\left(f(\cdot, x'')\right)(x')\,.
  \end{align*}
  We see thus that
  \begin{equation*}
    \left(\iota^{\ast}_{\RR^{d}} B f\right)(x') = \left(Bf\right)(x',0)
    = A\left( f(\cdot, 0) \right)(x') = \left(A\iota^{\ast}_{\RR^{d}}f\right)(x')\,,
  \end{equation*}
  which achieves the proof.
\end{proof}

In the following theorem, we prove the analogue of the above lemma for the case of a compact manifold $M$.

\begin{thm}
  Let $M$ be a compact submanifold of $\RR^{d_{0}}$ and $A$ a pseudo-differential operator on $M$ of class $\OPS{r}(M)$. There exists a pseudo-differential operator $B$ of class $\OPS{r}(\RR^{d_{0}})$, with compactly supported symbol, in the $x$-variable, such that
  \begin{equation*}
    \iota^{\ast}_{M} B = A \iota^{\ast}_{M}\,,
  \end{equation*}
  where $\iota_{M} : M \to \RR^{d_{0}}$ is the canonical embedding.
\end{thm}

\begin{proof}
  By the $2$-cluster property~\cite{Gru2009}, we can find a partition of unity $\{\Phi^{\prime}_j\}_{j = \overline{1,J_0}}$ associated with a cover of $M$, and a system of coordinate mappings $\{\kappa^{\prime}_{i}:\mathcal{U}^{\prime}_{i}\to U^{\prime}_{i}\}_{i = \overline{1,J_{1}}}$ such that every two functions $\Phi^{\prime}_{l},\Phi^{\prime}_m$ have their support in some $\mathcal{U}^{\prime}_{i}$, which will be denoted here by $\mathcal{U}^{\prime}_{lm}$. The sets $\mathcal{U}^{\prime}_{lm}$ may be identical for some different pairs. Now decompose the operator $A$ as:
  \begin{equation*}
    A = \sum_{l,m = 1}^{J_0}\Phi^{\prime}_{l} A \Phi_m^{\prime} \ \ ,
  \end{equation*}
  which is equivalent with:
  \begin{equation*}
    A = \sum_{i,j = 1}^{J_0}{\kappa^{\prime}_{lm}}^{*}(\varphi^{\prime}_{l}A_{lm} \varphi_m^{\prime})\ ,
  \end{equation*}
  where $A_{lm}\in \OPS{r}(\RR^{d})$ exist according to Definition \ref{def:Pscalar}.
  Because $M$ is an embedded submanifold there exist $\{\kappa_{i}:\mathcal{U}_{i}\to U_{i}\}_{i = \overline{1,J_{1}}}$ on $\RR^{d_0}$ such that $\kappa^{\prime}_{i} = \kappa_{i}|_{M}$ and $\mathcal{U}_{i}^{\prime} = \mathcal{U}_{i}\cap M$.
  Without loss of generality, we can assume that the partition of unity $\{\Phi^{\prime}_j\}_{j = \overline{1,J_0}}$ is constructed by restricting a portion of unity of the ambient space $\RR^{d_0}$ to the embedded manifold $M$. Thus we can extend every $\Phi_j^{\prime}$ trivially outside $M$ to $\Phi_j$ and define the operator:
  \begin{equation*}
    B: = \sum_{l,m = 1}^{J_0}{\kappa_{lm}}^{*}(\varphi_{l}B_{lm} \varphi_m)\ ,
  \end{equation*}
  where $B_{lm}$ is obtained according to Lemma~\ref{lem:trRd}, and therefore
  \begin{equation*}
    \iota^{\ast}_{\RR^{d}}B_{lm} = A_{lm}\iota^{\ast}_{\RR^{d}}.
  \end{equation*}
  Moreover, this construction defines a pseudo-differential operator in the right class and $B$ is a sum of pseudo-differential operators with x-compactly supported symbols. Further we make use of the identities
  $\iota^{\ast}_{\RR^{d}}(\varphi f) = \varphi\big|_{\RR^{d}} \iota^{\ast}_{\RR^{d}} f$,
  and $ \iota^{\ast}_{M} f\circ \kappa'^{-1} = \iota^{\ast}_{\RR^{d}}\left[ f\circ \kappa^{-1} \right]$, when $\mathrm{supp}\ f\subset U$ and $\kappa: U\rightarrow V$ is a diffeomorphism, in order to obtain:
  \begin{equation*}
    \iota^{\ast}_{M}(Bf) = A(\iota^{\ast}_{M} f), \ \ f\in C^{\infty}_c(\RR^{d_0}),
  \end{equation*}
  since:
  \begin{equation*}
    \iota^{\ast}_{\RR^{d}}\left[ \varphi\cdot f\circ \kappa^{-1} \right] = \varphi'\cdot \iota^{\ast}_{M} f\circ {\kappa^{\prime}}^{-1},
  \end{equation*}
  when $f\in C^{\infty}_c(\RR^{d_0})$.
\end{proof}

\begin{rem}
  Here $\iota_{M}^{\ast}$ denotes the trace operator, which is defined for continuous functions by restriction of the function to the submanifolds. Due to the Sobolev embedding theorem and our assumptions on $q_{0}$, we always work with continuous functions and thus this operator is well-defined and extends as a bounded operator to
  \begin{equation*}
    \iota_{M}^{\ast} : H^{q_{0}}(\RR^{d_{0}}) \to H^{q}(M)\,,
  \end{equation*}
  where $q_{0} = q + (d_{0}-d)/2$.
\end{rem}
It remains to extend the involved functions and diffeomorphisms. Therefore we will make use of the following extension operator, whose construction can be found in~\cite[Theorem~4.10]{GS2013}.

\begin{lem}
  Let $M \subset \RR^{d_{0}}$ be a compact submanifold of dimension $d$.
  There exists a continuous linear map
  \begin{equation*}
    E : C^{\infty}(M) \to C^{\infty}_c(\RR^{d_{0}})\,,
  \end{equation*}
  satisfying $\iota^{\ast}_{M} \circ E = \id_{C^{\infty}(M)}$,
  such that for all $q \geq 0$ and $q_{0} = q + (d_{0}-d)/2$, $E$ extends as a bounded operator to
  \begin{equation*}
    E : H^{q}(M) \to H^{q_{0}}(\RR^{d_{0}})\,.
  \end{equation*}
\end{lem}

\begin{cor}
  Given $q > d/2+1$ and $q_{0} = q + (d_{0} - d)/2$, there exists $\mathcal{U} \subseteq \mathcal{D}^{q}(M)$, an open neighborhood of the identity such that the map
  \begin{equation*}
    \mathcal{E} : \mathcal{U} \to \mathcal{D}^{q_{0}}(\RR^{d_{0}})\,\quad \mathcal{E}(\varphi) = E(\iota_{M} \circ \varphi - \iota_{M}) + \id_{\RR^{d_{0}}}\,,
  \end{equation*}
  is well-defined, smooth and satisfies $\mathcal{E}(\id_{M}) = \id_{\RR^{d_{0}}}$ as well as
  \begin{equation*}
    \iota_{M}^*\circ \mathcal{E}(\varphi) = \mathcal{E}(\varphi) \circ \iota_{M} = \iota_{M} \circ \varphi\,.
  \end{equation*}
\end{cor}

\begin{proof}
  We consider $\mathcal{D}^{q}(M)$ as a subset of $H^{q}(M,\RR^{d_{0}})$ and apply the extension operator $E$ component-wise. The map $\mathcal{E}$ is well-defined and continuous into $H^{q_{0}}(\RR^{d_{0}},\RR^{d_{0}})$ and satisfies $\mathcal{E}(\id_{M}) = \id_{\RR^{d_{0}}}$. Because $\mathcal{D}^{s_{0}}(\RR^{d_{0}})$ is open in $H^{q_{0}}(\RR^{d_{0}},\RR^{d_{0}})$, we can find a small neighborhood $\mathcal{U}$ of $\id_{M}$ such that $\mathcal{E}$ maps $\mathcal{U}$ into $\mathcal{D}^{q_{0}}(\RR^{d_{0}})$. The required identity for $\mathcal{E}$ follows from properties of $E$ as follows,
  \begin{align*}
    \mathcal{E}(\varphi) \circ \iota_{M}
     & = E(\iota_{M} \circ \varphi - \iota_{M}) \circ \iota_{M} + \id_{\RR^{d_{0}}} \circ \iota_{M} \\
     & = \iota_{M} \circ \varphi - \iota_{M} + \iota_{M} = \iota_{M} \circ \varphi\,.\qedhere
  \end{align*}
\end{proof}

\begin{lem}\label{lem:conj:extension}
  Let $A$ be a pseudo-differential operator of class $\OPS{r}(M)$ and $B$ a pseudo-differential operator of Hörmander's class $\OPS{r}(\RR^{d_{0}})$, such that $\iota^{\ast}_{M} B = A\hspace{0,1cm}\iota^{\ast}_{M}$. Let $\mathcal{E} : \mathcal{U} \to \mathcal{D}^{q_{0}}(\RR^{d_{0}})$ be as above. Then for each $\varphi \in \mathcal{U}$,
  \begin{equation*}
    A_{\varphi} = \iota^{\ast}_{M} B_{\mathcal{E}(\varphi)} E\,.
  \end{equation*}
\end{lem}

\begin{proof}
  We will use the following identities
  \begin{equation*}
    \iota^{\ast}_{M} R_{\mathcal{E}(\varphi)} = R_{\varphi} \iota^{\ast}_{M}, \qquad \iota^{\ast}_{M} R_{\mathcal{E}(\varphi)^{-1}} = R_{\varphi^{-1}} \iota^{\ast}_{M}\,.
  \end{equation*}
  To check the first identity, take $f \in H^{q_{0}}(\RR^{d_{0}})$. Then
  \begin{equation*}
    \iota^{\ast}_{M} R_{\mathcal{E}(\varphi)} f = f \circ \mathcal{E}(\varphi) \circ \iota_{M} = f \circ \iota_{M} \circ \varphi = R_{\varphi} \iota^{\ast}_{M} f\,.
  \end{equation*}
  The second identity follows from the first one by applying $R_{\varphi^{-1}}$ from the left and $R_{\mathcal{E}(\varphi)^{-1}}$ from the right. The lemma now follows from
  \begin{align*}
    \iota^{\ast}_{M} B_{\mathcal{E}(\varphi)} E
     & = \iota^*_{M} R_{\mathcal{E}(\varphi)} B R_{\mathcal{E}(\varphi)^{-1}} E \\
     & = R_{\varphi} \iota^*_{M} B R_{\mathcal{E}(\varphi)^{-1}} E              \\
     & = R_{\varphi} A \iota^*_{M} R_{\mathcal{E}(\varphi)^{-1}} E              \\
     & = R_{\varphi} A R_{\varphi^{-1}} \iota^*_{M} E = A_{\varphi}\,,
  \end{align*}
  since $\iota^*_{M}\circ E = \id_{H^{q}(M)}$.
\end{proof}

\begin{proof}[Proof of Theorem~\ref{thm:conj-functions}]
  Now the proof of Theorem~\ref{thm:conj-functions} follows directly from Theorem~\ref{thm:smoothness-of-conjugates} and Lemma~\ref{lem:conj:extension}
\end{proof}

\subsection{Conjugates of a pseudo-differential operator on a vector bundle}
\label{subsec:smoothness-conjugates-vector-bundle}

In the previous section, we have shown that for pseudo-differential operators acting on functions conjugation is smooth. By allowing matrix valued symbols this result generalizes directly to pseudo-differential operators acting on trivial bundles, i.e., the map
\begin{equation*}
  \left\{
  \begin{aligned}
    \DM{q}  & \to \mathcal{L}(H^{q}(M,\Rd),H^{q-r}(M,\Rd)) \\
    \varphi & \mapsto A_{\varphi} .
  \end{aligned}
  \right.
\end{equation*}
is smooth. For a pseudo-differential operator acting on mappings with values in a general vector bundle $E$ over a compact manifold it is not straightforward to define the analogous statement. Therefore, we first introduce spaces of $H^{s}$ sections over $\DM{q}$, see also \cite{Shk2000}:
\begin{equation*}
  H^{s}_{\mathcal{D}^{q}}(M,TM) = \set{v\in H^{s}(M,TM); \text{ such that } \pi \circ v = \varphi \in \DM{q}} .
\end{equation*}

This allows us to consider the conjugation of an operator $A$ acting on functions with values in the vector bundle $TM$:
\begin{equation*}
  \tilde{A}:
  \left\{
  \begin{aligned}
    \mathcal{D}^{q} & \to \mathcal{L}\left(H^{q}_{\mathcal{D}^{q}}(M,TM), H^{q-r}_{\mathcal{D}^{q}}(M,TM) \right) \\
    \varphi         & \mapsto A_{\varphi}
  \end{aligned}
  \right.
\end{equation*}
We have the following result concerning the smoothness of $\tilde{A}$.

\begin{thm}\label{thm:tildeA-smoothness}
  Let $A$ be a pseudo-differential operator of class $\OPS{r}(M,TM)$ with $r\geq 1$. Then $\tilde{A}$ is a smooth section of the vector bundle $\mathcal{L}\left(H^{q}_{\mathcal{D}^{q}}(M,TM),H^{q-r}_{\mathcal{D}^{q}}(M,TM)\right)$, for any $q>\frac{d}2+1$ with $q-r\geq0$.
\end{thm}

\begin{proof}
  Note first the following general observation: given two vector bundles $E$ and $F$ over a manifold $M$, we have an isomorphism between sections of $\mathcal{L}(E,F)$ and bundle mappings between $E$ and $F$. To prove this Lemma we will embed the manifold $M$ into $\RR^{d_{0}}$, which will allow us to use the result from Section~\ref{subsec:smoothness-conjugates-functions} for pseudo-differential operators acting on vector valued functions. We can then choose a vector bundle $NM$ over $M$ such that
  $TM \oplus NM \cong T\RR^{d_{0}}|_{M}$. Extend the operator $A$ to an operator $B$ acting on smooth
  sections of the trivial vector bundle $T\RR^{d_{0}}|_{M}$:
  \begin{align}
    B:\left\{
    \begin{aligned}
      \Gamma(T\RR^{d_{0}}|_{M}) & \to \Gamma(T\RR^{d_{0}}|_{M})                                           \\
      (v_{1} \oplus v_{2})      & \mapsto (A \oplus \Theta) (v_{1} \oplus v_{2}) : = Av_{1} \oplus \theta
    \end{aligned}
    \right.
  \end{align}
  where $\Theta$ associates to every $v_{2}\in\Gamma(NM)$ the zero section $\theta\in\Gamma(NM)$.

  It is easy to see that $B$ is also a pseudo-differential operator of class $\OPS{r}(T\RR^{d_0}|_{M})$. Consider an open set $\mathcal{U}$ on $M$ to which both vector bundles $TM$, $NM$ locally trivialize. Let $\Phi_{\mathcal{U}}:N\mathcal{U}\to \mathcal{U}\times \RR^{d_0-d}$ be a trivialization of $NM$, and $\kappa: \mathcal{U} \to U$, a coordinate mapping, then  $\Psi_{\mathcal{U}}:T\RR^{d_0}|_{\mathcal{U}}\to \mathcal{U}\times \RR^{d_0}$ is a trivialization of $T\RR^{d_0}|_{M}$:
  \begin{equation*}
    \Psi_{\mathcal{U}}(v_{1}(x) \oplus v_{2}(x)): = \left(x,T_x\kappa(v_{1}(x)) \oplus \Phi_{\mathcal{U}}(x)(v_{2}(x))\right), \quad x\in \mathcal{U}.
  \end{equation*}
  Similar to~\eqref{eq:pullback} and~\eqref{eq:pushforward} the push-forward and the pullback can be defined:
  \begin{equation*}
    {\chi}^{*}(u_{1} \oplus u_{2}) = \left({{\chi}^{TM}}\right)^{*}(u_{1}) \oplus \left({{\chi}^{NM}}\right)^{*}(u_{2})
  \end{equation*}
  \begin{equation*}
    {\chi}_{*}(v_{1} \oplus v_{2}) = \left({{\chi}^{TM}}\right)_*(v_{1}) \oplus \left({{\chi}^{NM}}\right)_*(v_{2})
  \end{equation*}
  We will use them to pullback operators like in Definition \ref{def:P-complex}. For two arbitrary functions $\Phi$, $\Psi$ in $C^{\infty}_c(\mathcal{U})$, we get the identity
  \begin{equation*}
    \Phi B \Psi = {\chi}^{*}\bigg(\varphi \left(A_{\raisebox{-1pt}{\tiny U}} \oplus \Theta_{\raisebox{-1pt}{\tiny U}}\right)\psi\bigg)\ ,
  \end{equation*}
  and thus $B\in \OPS{r}(T\RR^{d_0}|_{M})$. Associated with $B$ we have the induced mapping:
  \begin{equation*}
    \tilde{B}:
    \left\{
    \begin{aligned}
      \mathcal{D}^{q} & \to \mathcal{L}\left(H^{q}_{\mathcal{D}^{q}}\left(M,T\RR^{d_{0}}|_{M}\right), H^{q-r}_{\mathcal{D}^{q}}\left(M,T\RR^{d_{0}}|_{M}\right)\right) \\
      \varphi         & \mapsto B_{\varphi}
    \end{aligned}
    \right.
  \end{equation*}
  Now, let us introduce the mappings $\mathrm{I}$ and $\Pi$, defined by
  \begin{equation*}
    \mathrm{I}: H^{q}_{\mathcal{D}^{q}}(M,TM)\rightarrow H^{q}_{\mathcal{D}^{q}}(M,T\RR^{d_{0}}|_{M}),\quad \mathrm{I}(v) : = T\iota \circ v.
  \end{equation*}
  where $T\iota$ is the tangent map of the embedding $\iota:M\to\RR^{d_0}$, and
  \begin{equation*}
    \Pi: H^{q-r}_{\mathcal{D}^{q}}(M,T\RR^{d_{0}}|_{M})\rightarrow H^{q-r}_{\mathcal{D}^{q}}(M,TM),\quad \Pi(v) : = p \circ v,
  \end{equation*}
  where $p: T\RR^{d_0}|_{M}\to TM$ is the projection onto the first factor of the Whitney sum $TM \oplus NM$.

  The smoothness of the mapping $\tilde B$ follows from Theorem~\ref{thm:conj-functions}. The smoothness of $\mathrm{I}$ and $\Pi$ is a consequence of  Lemma A.5 in \cite{BHM2019}.
  Since the identities $$\mathrm{I}(v\circ\varphi^{-1}) = (\mathrm{I}v)\circ\varphi^{-1}\quad\text{ and }\quad\Pi (w\circ\varphi) = (\Pi w)\circ\varphi$$
  hold for all $v\in H^{q}_{\mathcal{D}^{q}}(M,TM)$, $w\in H^{q-r}(M,T\RR^{d_{0}}|_{M})$, and $\varphi\in \DM{q}$ we have
  \begin{equation*}
    \Pi_\varphi \tilde B \mathrm{I}_\varphi = \tilde A .
  \end{equation*}
  Thus, altogether we have proven the smoothness of the section $\tilde A$.
\end{proof}

\section{Smoothness of the metric and the spray on $\DM{q}$}
\label{sec:metric-and-spray}

In this Section we will study the smoothness of the extended metric and spray on the Hilbert manifold $\DM{q}$.

\subsection{Smoothness of the extended metric}

Let us first recall that a Riemannian metric $G$ on $\DM{q}$ is a smooth, symmetric, positive definite, covariant 2-tensor field on $\DM{q}$, i.e., for each $\varphi$ we have a symmetric, positive definite, bounded, bilinear form $G_{\varphi}$ on
$T_{\varphi}\DM{q}$ and, in any local chart $U$, the mapping
\begin{equation}
  \varphi \mapsto G_{\varphi},\qquad U \to \mathcal{L}^{2}_{\operatorname{Sym}}(E, \RR)
\end{equation}
is smooth. Here $E$ is a local trivialization of the tangent bundle of $\DM{q}$ over $U$. We can then also consider the bounded linear operator
\begin{equation}\label{eq:metric}
  \tilde{G}_{\varphi} : T_{\varphi}\DM{q} \to T^{*}_{\varphi}\DM{q},
\end{equation}
called the \emph{flat map}, which is defined by $\tilde{G}_{\varphi}(v) : = G_{\varphi}(v,\cdot)$.
The metric is called a \emph{strong Riemannian metric} if $\tilde{G}_{\varphi}$ is a topological linear isomorphism for every $\varphi \in \DM{q}$, whereas it is called a \emph{weak Riemannian metric}
if it is only injective for some $\varphi \in \DM{q}$.

\begin{thm}\label{thm:metric-smoothness}
  Let $A$ be a $L^{2}$-symmetric, positive definite pseudo-differential operator of class $\OPS{2s}(M,TM)$ where $s\geq \frac1{2}$ and let $q>\frac{d}2+1$.
  \begin{enumerate}
    \item If $q\geq 2s$, then, the right-invariant, weak Riemannian metric
          \begin{equation*}
            G_{\varphi}(v_{1},v_{2}) = \int_{M} \left( A_{\varphi} v_{1} \cdot v_{2} \right) \varphi^{*} d\mu \,,\qquad \forall v_{1},v_{2}\in T_{\varphi}\Diff(M)
          \end{equation*}
          defined on $\DiffM$, extends to a \emph{smooth, weak Riemannian metric} on the Banach manifold $\DM{q}$.
    \item If $q = s$ and $A = B^{2}$, where $B$ is a $L^{2}$-symmetric, positive definite pseudo-differential operator of class $\OPS{s}(M,TM)$, then, the right-invariant, weak Riemannian metric defined on $\DiffM$ extends to a \emph{smooth, strong Riemannian metric} on the Banach manifold $\DM{q}$.
  \end{enumerate}
\end{thm}

\begin{rem}
  The inertia operator of the fractional order Sobolev metric $A = (1+\Delta)^{s}$ satisfies the assumptions of the above theorem for $s \ge 1/2$.
\end{rem}

We will now give the proof of Theorem~\ref{thm:metric-smoothness}.

\begin{proof}[Proof of Theorem~\ref{thm:metric-smoothness}]
  Let $r=2s\geq 1$. To prove that $G$ extends to a smooth Riemannian metric on $\DM{q}$ it remains to show that $G$ depends smoothly on the foot point
  $\varphi$. Using the assumptions of item (1) and Theorem~\ref{thm:tildeA-smoothness} it follows that the mapping
  \begin{equation}
    \left\{
    \begin{aligned}
      T\DM{q} \times_{\DM{q}} T\DM{q} & \to H^{q-r}(M,\RR)               \\
      (v,w)                           & \mapsto A_{\varphi}(v) \cdot w .
    \end{aligned}
    \right.
  \end{equation}
  is smooth. Here we used that
  \begin{equation}
    (v,w) \mapsto v\cdot w, \qquad H^{q-r}(M,TM) \times H^{q}(M,TM) \to H^{q-r}(M,\RR)
  \end{equation}
  is smooth for $q > d/2+1$ and $q \geq r$. Using that the Radon--Nikodym derivative $\frac{\varphi^{*}\mu}{\mu}\in H^{q-1}(M)$ -- since $\varphi\in \DM{q}$ the result follows by the Sobolev multiplication Lemma~\ref{lem:pointwise-multiplication}.

  To prove item (2) we rewrite the metric as
  \begin{align*}
    G_{\varphi}(v_{1},v_{2}) & = \int_{M} \left( A_{\varphi} v_{1} \cdot v_{2} \right) \varphi^{*} d\mu \\ & = \int_{M} \left( (B_{\varphi})^{2} v_{1} \cdot v_{2} \right) \varphi^{*} d\mu = \int_{M} \left( B_{\varphi} v_{1} \cdot B_{\varphi} v_{2} \right) \varphi^{*} d\mu  .
  \end{align*}
  Using that
  \begin{equation}
    (v,w) \mapsto v\cdot w, \qquad L^{2}(M,TM) \times L^{2}(M,TM) \to L^1(M,\RR)
  \end{equation}
  is smooth, the smoothness of the metric follows similar as for item (1).
\end{proof}

\subsection{Smoothness of the extended spray}

We will now prove smoothness of the extended spray on $T\DM{q}$, when the inertia operator $A$ is in the class $\mathcal{E}^{2s}(M,TM)$, as defined in~\ref{def:class-Er}.

\begin{thm}\label{thm:smoothness-spray}
  Let $s \ge \frac1{2}$ and $q > 1 + d/2$, with $q\ge 2s$. Let $A$ be a pseudo-differential operator of class $\mathcal{E}^{2s}(M,TM)$. Then the geodesic spray of the extended metric~\eqref{eq:definition-metric} on $\DM{q}$ is a smooth vector field on $T\DM{q}$.
\end{thm}

\begin{proof}
  Let $r=2s\geq 1$. The proof below is a modified version of the arguments used in \cite{BEK2015}. The derivation of the spray in the case of a compact manifold was detailed in~\cite[Section 3.1]{Kol2017}, see also \cite{MP2010}. When $\DM{q}$ is not parallelizable, a few explanations are required first. In that case, we need to introduce first an auxiliary connection on $T\Diff(M)$, which is induced by the Levi-Civita connection $\nabla^{g}$ on $M$ and was first considered in~\cite{FG1989} and then in~\cite{Bru2018} (see also~\cite[Section 3]{Mod2014}). We shall denote this covariant derivative by $\overline{\nabla}$. The remarkable observation is that if $\varphi(t)$ is a path on $\Diff(M)$ and $v(t)$ is a vector field on $\Diff(M)$ defined along the path $\varphi(t)$, we have
  \begin{equation}\label{eq:functorial-property}
    (\overline{\nabla}_{{\partial_t}}v)(t,x) = \nabla^{g}_{\varphi_{t}(t,x)} v(\cdot, x).
  \end{equation}
  Moreover, one can recast the Euler-Arnold equation \eqref{eq:Diff-Euler-Arnold} in the form
  \begin{equation}\label{materialderivative}
    \overline\nabla_{{\partial_t}}\dot\varphi = S_{\varphi}(\dot{\varphi})= R_{\varphi} \circ S\circ
    R_{\varphi^{-1}} (\dot{\varphi}) ,
  \end{equation}
  where
  \begin{equation*}\label{eq:spray}
    S(u) = A^{-1} \left\{ [A,\nabla_{u}] u - (\nabla u )^{t} Au - (\dive u) Au \right\}, \qquad u\in \Vect(M).
  \end{equation*}
  Note, that this formula requires the invertibility of the operator $A$, which is guaranteed due to the ellipticity and positivity assumptions on the class of operators. Let us now identify the subbundle of $2$-velocities
  \begin{equation*}
    T^2\DM{q} = \set{\xi\in TT\DM{q} : \ \bar \pi_{TM}(\xi)=T\bar\pi_M(\xi)}
  \end{equation*}
  with the Whitney sum $T\DM{q}\oplus T\DM{q}$ via $(\varphi,\dot\varphi, \ddot\varphi)\to (\dot\varphi,\overline\nabla_{\dot \varphi}\dot\varphi)$, as in \cite{Mod2014}.  Equation \eqref{materialderivative} corresponds under this identification to the spray equation and the geodesic spray can be now interpreted as a bundle map
  \begin{equation*}
    v\to (v, S_{\varphi}(v)), \qquad T\DM{q}\to T\DM{q}\oplus T\DM{q}.
  \end{equation*}
  In this way we can argue, in an elegant manner, the smoothness of the extended spray by investigating the three summands in $S$ separately. The smoothness of the bundle map
  \begin{equation*}
    \tilde A:  v\to A_{\varphi}(v), \qquad H^{q}_{\mathcal{D}^{q}}(M,TM) \to H^{q-r}_{\mathcal{D}^{q}}(M,TM).
  \end{equation*}
  for $q>\frac{d}2+1$ and $s\ge r$ follows from Theorem~\ref{thm:tildeA-smoothness}. The smoothness of
  \begin{equation*}
    \tilde{A}^{-1} : v\mapsto A^{-1}_{\varphi}(v), \qquad H^{q-r}_{\mathcal{D}^{q}}(M,TM) \to T\DM{q}
  \end{equation*}
  follows from the same arguments as in~\cite{BEK2015}. The smoothness of the mappings
  \begin{equation*}
    v \mapsto Q^{k}_{\varphi}(v) = \left(R_{\varphi}\circ Q^{k} \circ R_{\varphi^{-1}}\right)(v), \qquad k = 2,3,
  \end{equation*}
  where
  \begin{equation*}
    Q^{2}(u) : = (\nabla u)^t Au, \quad \text{and} \quad Q^{3}(u) : = (\dive u) Au
  \end{equation*}
  when $u\in\Vect^{q}(M)$, follows from the same line of reasoning as in~\cite{BEK2015,Kol2017} or~\cite{Ebi1970}. It remains to show the smoothness of $v \mapsto Q^{1}_{\varphi}(v)$, where $Q^{1}(u) : = [\nabla_{u},A]u$. Therefore, we note that the covariant derivative $\overline{\nabla}$ extends to a smooth covariant derivative $\widetilde\nabla$ on the vector bundle
  \begin{equation*}
    \mathcal{L}(H^q_{\DM{q}},H^{q-r}_{\DM{q}}),
  \end{equation*}
  see~\cite{BHM2019}.
  This allows us to identify the term $Q_{1}$ as the first derivative of $\tilde{A}$ since we have
  \begin{equation}
    (\widetilde\nabla_v \tilde{A})v = \overline\nabla_v(\tilde{A} v) - \tilde{A}(\overline \nabla_v v) = Q^{1}_{\varphi}(v).
  \end{equation}
  Thus the smoothness of $Q^{1}_{\varphi}(v)$ follows from the smoothness of $\tilde{A}$, which concludes the proof.
\end{proof}

\begin{rem}
  Note, that for strong Riemannian metrics the smoothness of the spray follows automatically, i.e., for metrics that satisfy the assumptions of item (2) in Theorem~\ref{thm:metric-smoothness}, we obtain the smoothness of the spray on $\DM{q}$ for $q = s$. We will later use this to obtain a global existence result.
\end{rem}

\section{Local and global well-posedness of the EPDiff equation}
\label{sec:well-posedness}

In this section we will prove our main theorem concerning local and global well-posedness properties of the geodesic equation.

\subsection{Local and global well-posedness in the Sobolev category}

We will first formulate the result in the Sobolev category and will later see, that many of the properties continue to hold in the smooth category.

\begin{thm}\label{thm:smooth-flow-Hq}
  Let $s \ge \frac12$ and $q > 1 + d/2$ with $q \ge 2s$. Let $A$ be a pseudo-differential operator in the class $\mathcal{E}^{2s}(M,TM)$, defined on the tangent bundle $TM$ of a compact manifold $M$ and let $G$ be the right invariant metric induced by $A$.
  We have:
  \begin{enumerate}
    \item The geodesic equations of the metric $G$ on $\DM{q}$ are locally well-posed, i.e., given any $v_{0}\in T\DM{q}$, there exists a unique non-extendable geodesic
          \begin{equation*}
            v \in C^{\infty}(J,T\DM{q})
          \end{equation*}
          defined on some open interval $J$, which contains $0$ and such that $v(0) = v_{0}$.
    \item The corresponding Euler-Arnold equation has, for any initial data $u_{0}\in \Vect^{q}(M)$, a unique non-extendable smooth solution
          \begin{equation*}
            u\in C^{0}(J,\Vect^{q}(M)) \cap C^{1}(J,\Vect^{q-1}(M))
          \end{equation*}
          defined on $J$.
  \end{enumerate}
\end{thm}

\begin{proof}
  $(1)$ follows directly from the Picard Lindelöf theorem (or Cauchy-Lipschitz theorem) on Banach manifolds, using that the geodesic spray is a smooth vector field on $T\DM{q}$ as shown in Theorem~\ref{thm:smoothness-spray}.

  To prove $(2)$, let $u_{0}\in \Vect^{q}(M)$. Then by $(1)$, there exists a curve
  \begin{equation*}
    v \in C^{\infty}(J,T\DM{q})
  \end{equation*}
  defined on some maximal time interval $J$ which contains $0$ and such that $v(0) = u_{0}$. Consider now the Eulerian velocity
  \begin{equation*}
    u(t) = v(t) \circ \varphi(t)^{-1}
  \end{equation*}
  where $\pi: T\DM{q} \to \DM{q}$ is the canonical projection and $\varphi(t) = \pi (v(t))$. Note that
  \begin{equation*}
    u\in C^{0}(J,\Vect^{q}(M)),
  \end{equation*}
  because the mapping
  \begin{equation*}
    (\varphi,v) \mapsto v \circ \varphi, \qquad \DM{q} \times T\DM{q} \to T\DM{q}
  \end{equation*}
  and the inversion
  \begin{equation*}
    \varphi \mapsto \varphi^{-1}, \qquad \DM{q} \to \DM{q}
  \end{equation*}
  are continuous, see~\cite[Theorem 1.2]{IKT2013}. Moreover, the mappings
  \begin{equation*}
    (\varphi,v) \mapsto v \circ \varphi, \qquad \DM{q-1} \times T\DM{q} \to T\DM{q-1}
  \end{equation*}
  and
  \begin{equation*}
    \varphi \mapsto \varphi^{-1}, \qquad \DM{q} \to \DM{q-1}
  \end{equation*}
  are $C^{1}$, see~\cite[Theorem 1.2]{IKT2013}.
\end{proof}

The right invariance of the metric even allows us to obtain global well-posedness of the geodesic equation and completeness of the corresponding metric space.

\begin{thm}
  Let $A = B^{2}$, where $B$ is of class $\mathcal{E}^{s}(M,TM)$ with $s > 1 + d/2$. Then:
  \begin{enumerate}
    \item The geodesic equations on $\DM{s}$ and the corresponding Euler-Arnold equation are globally well-posed.
    \item The space $\DM{s}$ equipped with the geodesic distance of the metric $G$ is a complete metric space.
    \item Any two elements in the connected component of the identity in $\DM{s}$ can be connected by a minimizing geodesic.
  \end{enumerate}
\end{thm}

\begin{proof}
  This result follows directly from~\cite[Corollary 7.6]{BV2017}, using that the metric extends to a smooth, strong, right-invariant metric by Theorem~\ref{thm:metric-smoothness}.
\end{proof}

In the following we will show that some of the well-posedness and completeness statements continue to hold in the smooth category.

\subsection{Local and global well-posedness in the smooth category}
\label{sec:globalwellposedness_smooth}

In their seminal article~\cite{EM1970} Ebin and Marsden made the remarkable observation that, due to the right-invariance of the metric, the maximal interval of existence is independent of the parameter $q$. This enables us to avoid Nash--Moser type schemes to prove local existence of smooth geodesics in the smooth category.

\begin{cor}\label{cor:smooth_flow}
  Let $s \ge \frac12$ and $A$ be a pseudo-differential operator in the class $\mathcal{E}^{2s}(M,TM)$, defined on the tangent bundle $TM$ of a compact manifold $M$ and let $G$ be the right invariant metric induced by $A$. Then:
  \begin{enumerate}
    \item The geodesic equations of the metric $G$ on $\DiffM$ are locally well-posed, \textit{i.e.}, given any $v_{0}\in T\DiffM$, there exists a unique non-extendable geodesic
          \begin{equation*}
            v \in C^{\infty}(J,T\DiffM)
          \end{equation*}
          defined on some maximal open time interval $J$, which contains $0$;
    \item The corresponding Euler-Arnold equation has, for any initial data $u_{0}\in \Vect(M)$, a unique non-extendable smooth solution
          \begin{equation*}
            u \in C^{\infty}(J,\Vect(M))
          \end{equation*}
          defined on $J$.
  \end{enumerate}
\end{cor}

\begin{proof}
  $(1)$ follows from point $(1)$ of Theorem~\ref{thm:smooth-flow-Hq}, and the invariance of the spray; this is known as the \emph{no-loss-no-gain lemma}~\cite[Theorem 12.1]{EM1970}.

  In the smooth category, the mappings
  \begin{equation*}
    (\varphi,v) \mapsto v \circ \varphi, \qquad \DiffM \times T\DiffM \to T\DiffM
  \end{equation*}
  and
  \begin{equation*}
    \varphi \mapsto \varphi^{-1}, \qquad \DiffM \to \DiffM
  \end{equation*}
  are smooth, see~\cite{IKT2013}, which proves $(2)$.
\end{proof}

It is clear, that one can never hope for metric completeness of a Sobolev type metric on the space of smooth diffeomorphisms.
However, the geodesic completeness results of the previous section still hold in the smooth category.

\begin{cor}\label{cor:globalexistence}
  Let $A = B^{2}$, where $B$ is of class $\mathcal{E}^{s}(M,TM)$ with $s>\frac{d}2+1$. Then, the geodesic equations on $\DiffM$ and the corresponding Euler-Arnold equation are globally well-posed.
\end{cor}

\begin{proof}
  To prove this result we need a slightly improved version of the no-loss-no-gain result. The reason is that using our results from Theorem~\ref{thm:smooth-flow-Hq}, we only know the smoothness of the geodesic spray on $T\DM{q}$ for $q = s$ and for $q\geq 2s$, i.e., not for $q\in(s,2s)$. To conclude the global existence using~\cite[Theorem 12.1]{EM1970} in the smooth category one needs the smoothness of the spray for all $q \geq s$ (or at least for all $s+k$ with $k\in\NN$). However, it has been shown in~\cite[Corollary 4.5.]{Bru2017}, that the smoothness and $\DM{r}$-equivariance of the spray on $\DM{s}$ already imply the smoothness of the spray on $T\DM{\alpha}$ for all $\alpha = s+k$ with $k\in \NN$ and thus the result follows.
\end{proof}

\begin{rem}
  Note that the results of this Section apply in particular to the $H^{s}$-metric for $s > 1/2$ (and $s > 1 + d/2$ for global well-posedness, respectively).
\end{rem}

\appendix

\section{A Sobolev boundedness theorem}
\label{sec:Sobolev-boundedness}

The goal of this appendix is to prove the following theorem.

\begin{thm}\label{thm:boundedness-lemma}
  Let $P \in \OPS{r+n-1}$ with a Hermitian symbol compactly supported in $x$. Given $w \in \mathrm{C}_c^{\infty}(\RR^{d},\RR^{d})$ and $f_{1}, \dotsc , f_{n} \in \mathrm{C}_c^{\infty}(\RR^{d},\RR)$, we have:
  \begin{equation*}
    \norm{S_{n,P}(f_{1},\dotsc , f_{n})w}_{H^{q-r}} \lesssim \norm{f_{1}}_{H^{q}} \dotsb \norm{f_{n}}_{H^{q}} \norm{w}_{H^{q-1}},
  \end{equation*}
  for $q > 1 + d/2$ and $r \le q$, where:
  \begin{equation*}
    S_{n,P}(f_{1},f_{2}, \dotsc , f_{n}) : = [f_{1},[f_{2} \dotsb [ f_{n}, P]\dotsb]].
  \end{equation*}
\end{thm}

The proof we present here is inspired from~\cite{EK2014,BEK2015} which was given for a Fourier multiplier but requires a trick for a pseudo-differential operator which is used to prove the $L^{2}$ boundedness theorem for an operator in $\OPS{0}$ (see~\cite[Part II, Section 2.4]{RT2010}). Because $p(x,\xi)$ is compactly supported in $x$, we can use the Fourier transform of $p(x,\xi)$ with respect to $x:$
\begin{equation*}
  \hat{p}(\lambda,\xi) : = \int_{\Rd} e^{-2\pi i \langle x,\lambda \rangle} p(x, \xi) \, dx,
\end{equation*}
which allows to rewrite $P$ as
\begin{equation}\label{eq:compact-support-case}
  (Pw)(x) = \int_{\Rd} e^{2\pi i \langle x,\lambda \rangle} \left(\hat{p}(\lambda,D)w\right)(x) \, d \lambda,
\end{equation}
where $\hat{p}(\lambda,D)$ is a \emph{Fourier multiplier} with symbol $\hat{p}(\lambda,\xi)$. Before entering the details of the proof of Theorem~\ref{thm:boundedness-lemma}, we will establish the following lemma.

\begin{lem}\label{lem:multi-symbol-expression}
  Let $P \in \OPS{r}$ with a Hermitian symbol $p(x,\xi)$ compactly supported in $x$ and let $f_{1}, \dotsc , f_{n} \in \CScR$, then for each $n \ge 1$ we have
  \begin{equation}\label{eq:multi-symbol-expression}
    \left(S_{n,P}(f_{1},\dotsc , f_{n})w\right)(x) = \int_{\Rd} e^{2\pi i \langle x,\lambda \rangle} (P_{n}(\lambda)w)(x) \, d \lambda,
  \end{equation}
  where
  \begin{equation*}
    \mathfrak{F}(P_{n}(\lambda)w)(\xi) = \int_{\xi_{0} + \dotsb +\xi_{n} = \xi} \hat{f}_{1}(\xi_{1}) \dotsb \hat{f}_{n}(\xi_{n}) \cdot \hat{p}_{n}(\lambda,\xi_{0},\xi_{1},\dotsc \xi_{n})\hat{w}(\xi_{0}) \, d\mu ,
  \end{equation*}
  $d\mu$ is the Lebesgue measure on the subspace $\xi_{0} + \dotsb + \xi_{n} = \xi$ of $(\RR^{d})^{n+1}$ and
  \begin{equation*}
    \hat{p}_{n}(\lambda,\xi_{0},\xi_{1}, \dotsc ,\xi_{n}) : = \sum_{J \subseteq \set{1, \dotsc , n}}(-1)^{\abs{J}} \hat{p}\left( \lambda,\xi_{0} + \sum_{j\in J}\xi_{j} \right)
  \end{equation*}
\end{lem}

\begin{proof}
  The proof is achieved by induction on $n$. For $n = 1$, using~\eqref{eq:compact-support-case}, we get
  \begin{equation*}
    [f_{1},P]w(x) = \int_{\Rd} e^{2\pi i \langle x,\lambda \rangle} \left(P_{1}(\lambda)w\right)(x) \, d \lambda,
  \end{equation*}
  where
  \begin{equation*}
    P_{1}(\lambda) : = [f_{1},\hat{p}(\lambda,D)].
  \end{equation*}
  Therefore, \eqref{eq:multi-symbol-expression} is true for $n = 1$ with
  \begin{equation*}
    \hat{p}_{1}(\lambda,\xi_{0},\xi_{1}) = \hat{p}(\lambda,\xi_{0}) - \hat{p}(\lambda,\xi_{0}+\xi_{1}).
  \end{equation*}
  Suppose now that~\eqref{eq:multi-symbol-expression} is true for some $n \ge 1$. Using the fact that
  \begin{equation*}
    S_{n+1,P}(f_{1},\dotsc , f_{n+1}) = [f_{n+1}, S_{n,P}(f_{1},\dotsc , f_{n})],
  \end{equation*}
  we get
  \begin{equation*}
    (S_{n+1,P}(f_{1},\dotsc , f_{n+1})w)(x) = \int_{\Rd} e^{2\pi i \langle x,\lambda \rangle} \left([f_{n+1},P_{n}(\lambda)]w\right)(x) \, d \lambda .
  \end{equation*}
  But, by the \emph{coarea formula} and the recurrence hypothesis, we have
  \begin{multline*}
    \mathfrak{F}([f_{n+1},P_{n}(\lambda)]w)(\xi) = \int_{\xi_{0} + \dotsb +\xi_{n+1} = \xi} \hat{f}_{1}(\xi_{1}) \dotsb \hat{f}_{n+1}(\xi_{n+1}) \cdot
    \\
    \Big[ \hat{p}_{n}(\lambda,\xi_{0},\dotsc,\xi_{n}) - \hat{p}_{n}(\lambda, \xi_{0} + \xi_{n+1}, \xi_{1}, \dotsc , \xi_{n}) \Big]\hat{w}(\xi_{0}) \, d\mu .
  \end{multline*}
  Thus, it remains to show that
  \begin{equation*}
    \hat{p}_{n}(\lambda,\xi_{0},\dotsc,\xi_{n}) - \hat{p}_{n}(\lambda,\xi_{0}+\xi_{n+1}, \xi_{1},\dots,\xi_{n}) = \hat{p}_{n+1}(\lambda,\xi_{0},\dotsc,\xi_{n+1}).
  \end{equation*}
  To do this, take the formula for $\hat{p}_{n+1}$ and split the sum as follows
  \begin{multline*}
    \hat{p}_{n+1}(\lambda, \xi_{0}, \dotsc , \xi_{n+1}) = \sum_{J \subseteq \set{1, \dotsc , n}} (-1)^{\abs{J}} \hat{p}\left(\lambda, \xi_{0} + \sum_{j \in J} \xi_{j}\right)
    \\
    - \sum_{J \subseteq \set{1, \dotsc , n}} (-1)^{\abs{J}} \hat{p}\left(\lambda, \xi_{0} + \xi_{n+1} + \sum_{j \in J} \xi_{j}\right),
  \end{multline*}
  which is equal to
  \begin{equation*}
    \hat{p}_{n}(\lambda,\xi_{0},\dotsc,\xi_{n}) - \hat{p}_{n}(\lambda,\xi_{0}+\xi_{n+1}, \xi_{1},\dotsc,\xi_{n}),
  \end{equation*}
  and achieves the proof.
\end{proof}

Next, we will provide an estimate on $\hat{p}_{n}$.

\begin{lem}\label{lem:pn-estimate}
  Suppose that $p\in \mathbf{S}^{r+n-1}(\Rd \times \Rd, \mathrm{M}_{d}(\CC))$ is compactly supported in $x$ and that $r\geq 1$, then the following estimate holds:
  \begin{multline}\label{eq:pn-estimate}
    \abs{\hat{p}_{n}(\lambda,\xi_{0},\xi_{1},\ldots \xi_{n})}
    \\
    \le C_{p,N} (1 + \abs{\lambda})^{-N} \left( \prod_{j = 1}^{n} \langle \xi_{j} \rangle \right)
    \sum_{J \subseteq \set{1, \dotsc , n}} \left\langle \xi_{0} + \sum_{j \in J} \xi_{j} \right\rangle^{r-1},
  \end{multline}
  for all $N \ge 0$, where $C_{p,N}>0$ depends only on $p$ and $N$.
\end{lem}

\begin{proof}
  Fix $\xi_{0},\dots,\xi_{n} \in \RR^{d}$. Let $K_{0} = \{ \xi_{0}\}$ and we define for $k = 1,\dots,n$ the set $K_{k}$ to be the convex hull of the sets $K_{k-1}$ and $K_{k-1} + \xi_{n+1-k}$. Then $K_{k}$ is the convex hull of the points $\xi_{0} + \sum_{j \in J} \xi_j$, where $J$ is any subset of $\{n+1-k,\dots,n\}$. Let $F_{k}$ be the sequence of mappings defined inductively by
  \begin{equation*}
    F_{0}(\lambda,\xi) = \hat{p}(\lambda,\xi), \quad F_{k}(\lambda,\xi) = F_{k-1}(\lambda,\xi) - F_{k-1}(\lambda,\xi + \xi_{k}),
  \end{equation*}
  for $k = 1,\dotsc,n$. In particular, we have
  \begin{equation*}
    F_{n}(\lambda,\xi_{0}) = \hat{p}_{n}(\lambda,\xi_{0},\xi_{1},\ldots \xi_{n}).
  \end{equation*}
  Then we can apply the mean value theorem to the recurrence relation to obtain
  \begin{equation*}
    \abs{F_{k}(\lambda,\xi)} = \abs{F_{k-1}(\lambda,\xi) - F_{k-1}(\lambda,\xi+\xi_{k})} \le \abs{\xi_{k}} \cdot \sup_{\eta \in K_{n-k+1}}\abs{\nabla^{\xi}F_{k-1}(\lambda,\eta)},
  \end{equation*}
  for all $\xi \in K_{n-k}$ and $ 1 \le k \le n$ and where $\nabla^{\xi}$ is the differential relative to the second variable $\xi$. Since the recurrence relation for $F_{k}$ is linear, it remains valid for all derivatives of $F_{k}$, i.e.,
  \begin{equation*}
    \sup_{\xi \in K_{n-k}}\abs{\left(\nabla^{\xi}\right)^{j}F_{k}(\lambda,\xi)} \le \abs{\xi_{k}} \cdot \sup_{\eta \in K_{n-k+1}}\abs{\left(\nabla^{\xi}\right)^{j+1}F_{k-1}(\lambda,\eta)},
  \end{equation*}
  for all $j\in \NN$. Starting with $k = n$ and applying the estimate iteratively we obtain
  \begin{multline*}
    \abs{F_{n}(\lambda,\xi_{0})}
    \le \abs{\xi_{n}} \cdot \sup_{\xi \in K_{1}}\abs{\nabla^{\xi}F_{n-1}(\lambda,\xi)}
    \le \dotsb
    \\
    \dotsb \le \left( \prod_{j = 1}^{n} \abs{\xi_{j}} \right) \cdot \sup_{\xi \in K_{n}} \abs{\left(\nabla^{\xi}\right)^{n} \hat{p}(\lambda,\xi)} \, .
  \end{multline*}
  Now, given $\alpha, \beta \in \NN^{d}$, we have
  \begin{equation*}
    (2\pi i \lambda)^{\alpha}\partial_{\xi}^{\beta} \hat{p}(\lambda,\xi) = \int_{\Rd} e^{-2\pi i \langle x,\lambda \rangle} \partial_{x}^{\alpha}\partial_{\xi}^{\beta} p(x,\xi) \, dx \, ,
  \end{equation*}
  from which we deduce that for every $N \ge 0$, there exists a constant $C_{p,N} > 0$ such that
  \begin{equation*}
    \abs{\left(\nabla^{\xi}\right)^{n} \hat{p}(\lambda,\xi)} \le C_{p,N} (1 + \abs{\lambda})^{-N} \langle \xi \rangle^{r-1}
  \end{equation*}
  and it remains to estimate $\langle \xi \rangle^{r-1}$ on the set $K_{n}$. For $r \geq 1$ the function $\xi \mapsto \langle \xi \rangle^{r-1}$ is convex and so it attains its maximum at one of the points $\xi_{0} + \sum_{j \in J} \xi_{j}$. Hence
  \begin{equation*}
    \sup_{\xi \in K_{n}} \left\langle \xi \right\rangle^{r-1} \leq
    \sum_{J \subseteq \set{1, \dotsc , n}} \left\langle \xi_{0} + \sum_{j \in J} \xi_{j} \right\rangle^{r-1}\,,
  \end{equation*}
  which achieves the proof, because $\abs{\xi_{j}} \le \langle \xi_{j} \rangle$.
\end{proof}

\begin{proof}[Proof of Theorem~\ref{thm:boundedness-lemma}:]
  By Lemma~\ref{lem:multi-symbol-expression}, we get
  \begin{equation*}
    \mathfrak{F}\left(S_{n,P}(f_{1},\dotsc , f_{n})w\right)(\xi) = \int_{\Rd} e^{2\pi i \langle x,\lambda \rangle}
    \mathfrak{F}\left(P_{n}(\lambda)w\right)(\xi) \, d\lambda \, ,
  \end{equation*}
  and thus
  \begin{multline*}
    \abs{\mathfrak{F}\left(S_{n,P}(f_{1},\dotsc , f_{n})w\right)(\xi)}
    \\
    \le \int_{\Rd} \int_{\xi_{0} + \dotsb +\xi_{n} = \xi} \abs{\hat{f}_{1}(\xi_{1}) \dotsb \hat{f}_{n}(\xi_{n}) \cdot \hat{p}_{n}(\lambda,\xi_{0},\xi_{1},\dotsc \xi_{n})\hat{w}(\xi_{0})} \, d\mu \, d\lambda \, .
  \end{multline*}
  where $d\mu$ is the Lebesgue measure on the subspace $\xi_{0}+\xi_{1}+\dotsc +\xi_{n} = \xi$ of $(\RR^{d} )^{n+1}$. Now, by Lemma~\ref{lem:pn-estimate}, we get
  \begin{multline*}
    \abs{\mathfrak{F}\left(S_{n,P}(f_{1},\dotsc , f_{n})w\right)(\xi)} \le C_{p,N} \left( \int_{\Rd} (1 + \abs{\lambda})^{-N} \, d\lambda \right) \sum_{J \subseteq \set{1, \dotsc , n}}
    \\
    \int_{\xi_{0} + \dotsb +\xi_{n} = \xi} \left\langle \xi_{0} + \sum_{j \in J} \xi_{j} \right\rangle^{r-1} \left( \prod_{j = 1}^{n} \langle \xi_{j} \rangle \right) \abs{\hat{f}_{1}(\xi_{1}) \dotsb \hat{f}_{n}(\xi_{n})} \abs{\hat{w}(\xi_{0})} \, d\mu \, .
  \end{multline*}
  But
  \begin{multline*}
    \int_{\xi_{0} + \dotsb +\xi_{n} = \xi} \left\langle \xi_{0} + \sum_{j \in J} \xi_{j} \right\rangle^{r-1} \left( \prod_{j = 1}^{n} \langle \xi_{j} \rangle \right) \abs{\hat{f}_{1}(\xi_{1}) \dotsb \hat{f}_{n}(\xi_{n})} \abs{\hat{w}(\xi_{0})} \, d\mu
    \\
    = \mathfrak{F} \left( \prod_{j \in J^{c}} \mathfrak{F}^{-1}\left(\abs{\langle \xi_{j} \rangle \hat{f}_{j}}\right) \Lambda^{r-1} \left[ \prod_{j \in J} \mathfrak{F}^{-1}\left(\abs{\langle \xi_{j} \rangle \hat{f}_{j}}\right) \mathfrak{F}^{-1}\left(\abs{\hat{w}}\right) \right] \right) (\xi) \, ,
  \end{multline*}
  where $\Lambda^{s}$ is the Fourier multiplier with symbol $\langle \xi \rangle^{s}$. We have thus, using the Plancherel identity and taking $N > d$
  \begin{align*}
     & \norm{S_{n,P}(f_{1},\dotsc , f_{n})w}_{H^{q-r}} = \norm{\langle \xi \rangle^{q-r} \mathfrak{F}(S_{n,P}(f_{1},\dotsc , f_{n})w)}_{L^{2}}
    \\
     & \quad \lesssim \sum_{J \subseteq \set{1, \dotsc , n}} \norm{ \prod_{j \in J^{c}} \mathfrak{F}^{-1}\left(\abs{\langle \xi_{j} \rangle \hat{f}_{j}}\right)\Lambda^{r-1} \left[ \prod_{j \in J} \mathfrak{F}^{-1}\left(\abs{\langle \xi_{j} \rangle \hat{f}_{j}}\right) \mathfrak{F}^{-1}\left(\abs{\hat{w}}\right) \right] }_{H^{q-r}}
    \\
     & \quad\lesssim \sum_{J \subseteq \set{1, \dotsc , n}} \norm{ \prod_{j \in J^{c}} \mathfrak{F}^{-1}\left(\abs{\langle \xi_{j} \rangle \hat{f}_{j}}\right)}_{H^{q-1}} \norm{\Lambda^{r-1} \left[ \prod_{j \in J} \mathfrak{F}^{-1}\left(\abs{\langle \xi_{j} \rangle \hat{f}_{j}}\right) \mathfrak{F}^{-1}\left(\abs{\hat{w}}\right) \right] }_{H^{q-r}}
    \\
     & \quad \lesssim \norm{f_{1}}_{H^{q}} \dotsb \norm{f_{n}}_{H^{q}} \norm{w}_{H^{q-1}}.
  \end{align*}
\end{proof}


\end{document}